\documentclass[12pt]{article}

\usepackage{bm}
\usepackage[extra]{tipa}
\usepackage{amsmath}
\usepackage{amssymb}
\usepackage{amsthm}
\usepackage{mathtools}
\usepackage{authblk}

\usepackage{enumitem}
\usepackage{natbib}
\usepackage{hyperref}
\usepackage{graphicx}
\usepackage{xr}
\usepackage{url}
\usepackage{xcolor}

\usepackage{a4wide}

\newcommand{\iu}{\mathit{i}}

\newcommand{\norm}[1]{\lVert#1\rVert}
\newcommand{\abs}[1]{\lvert#1\rvert}

\newcommand{\bE}[1]{\mathbf{E}[#1]}
\newcommand{\bP}[1]{\mathbf{P}(#1)}
\newcommand{\bV}[1]{\mathbf{Var}[#1]}

\DeclarePairedDelimiterX{\inp}[2]{\langle}{\rangle}{#1, #2}

\newtheorem{theorem}{Theorem}
\newtheorem{proposition}{Proposition}
\newtheorem{corollary}{Corollary}
\newtheorem{lemma}{Lemma}

\newcommand{\N}{\mathbb{N}}
\newcommand{\R}{\mathbb{R}}
\newcommand{\Z}{\mathbb{Z}}
\newcommand{\I}{\mathbf{I}}

\graphicspath{
  {./figures/}
}

\begin{document}

\title{Estimation of L\'evy-driven CARMA models under renewal sampling}

\author[1]{Frank Bosserhoff}
\author[2]{Giacomo Francisci}
\author[1]{Robert Stelzer}

\affil[1]{Institute of Mathematical Finance, Ulm University, Ulm, Germany \texttt{robert.stelzer@uni-ulm.de}}

\affil[2]{Department of Mathematics, University of Trento, Trento, Italy \texttt{giacomo.francisci@unitn.it}}

\date{\today}

\maketitle

\begin{abstract}
  Continuous-time autoregressive and moving average (CARMA) models are extensively used to model high-frequency and irregularly sampled data. We study Whittle estimation for the model parameters when the process is observed at renewal times. The driving noise is assumed to be a L\'evy process allowing for more flexibility including heavy-tailed marginal distributions and jumps in the sample paths. We show that the Whittle estimator based on the integrated periodogram is consistent and asymptotically normal under very mild conditions. To obtain these results, we establish the asymptotic normality of the integrated periodogram. \\

\end{abstract}

\section{Introduction}
\label{sec:introduction}

Continuous-time autoregressive moving average (CARMA) processes have long attracted interest by researchers in physics, engineering, probability and statistics \citep{Fowler-1936,Doob-1944}. Renewed interest in these models has emerged due to advances in stochastic modeling and the increased availability of high-frequency and irregularly spaced data \citep{Brockwell-2001a,Brockwell-2001b,Brockwell-2009,Brockwell-2014, Brockwell-2024}. Among many fields, they are applied in finance and economics to model asset and option prices, interest rates and volatility \citep{Brockwell-2009,Brockwell-2012,Andresen-2014,Benth-2014}; in signal processing to extract underlying signals from noisy data (e.g.\ \citet{McElroy-2013}); and in medical and natural sciences among others to forecast and model temperature, wind speed, light curves, and fluid turbulence \citep{Kelly-2014,Tomasson-2015,Yu-2021, Darus-2022}. In particular, the use of L\'evy processes as driving noise provides more flexibility and allows for heavy-tailed and asymmetric distributions and jumps in the sample paths, which are typically observed in many applications (see for instance \citet{Barndorff-Nielsen-1997,Rydberg-1997,Cont-2001}). Additionally, in a wide range of contexts, data are observed at random times, which are often considered independent of the value of the process \citep{Jones-1981,Jones-1985,Charlot-2008,Bardet-2010,Li-2014,Philippe-2021,Brandes-2023}. For instance, health monitoring systems such as smartphones and smartwatches allow to monitor a wealth of data including body temperature, electrocardiogram, blood oxygen, heart rate, breathing, as well as, noise levels, location, pace and speed. Measurements are often taken every few seconds or minutes when at rest to save battery. As another case, financial and economic data are in many cases sampled at irregular and randomly spaced times, which are often modeled as independent of the underlying process \citep{Bergstrom-1990,Ait-Sahalia-2003,Ait-Sahalia-2004, Hayashi-2005}. These considerations highlight the need for a flexible estimation method.

A number of papers in the existing literature deal with the problem of parameter estimation for univariate and multivariate L\'evy-driven CARMA models. \citet{Brockwell-2007},  \citet{Schlemm-2012}, and \citet{Brockwell-2019} establish consistency and asymptotic normality of a quasi-maximum likelihood estimator observed at equidistant time points. \citet{Fasen-2013} propose an estimation procedure based on convergence in probability of a smoothed normalized periodogram when the time step approaches zero. \citet{Fasen-2020} derive an indirect procedure via generalized M-estimators that is robust against outliers. More recently, \citet{Fasen-2022} establish consistency and asymptotic normality of Whittle parameter estimators. Their methodology uses that the discretely sampled process is a discrete-time vector ARMA process and it is not applicable to irregular sampling schemes. Furthermore, the equally spaced sampling scheme employed in the above papers is prone to identification issues caused by the aliasing effect and requires suitable identifiability assumptions.

In this paper, we investigate the problem of estimating the parameters of univariate CARMA processes when sampled at an independent renewal sequence. Section \ref{sec:main_definitions_results} contains a precise definition of the process and the sampling scheme. Incidentally, the use of an independent renewal sequence is beneficial for avoiding aliasing \citep{Masry-1978,Lii-1992}. Our anti-aliasing assumption is readily verified when the inter-arrival times are exponentially distributed. The estimation method uses that a certain integral of the spectral density of the sampled process is maximized at the true parameter value. The spectral density is estimated by the periodogram and the estimation procedure relies on maximizing the integrated periodogram \citep{Lii-1992}. We establish the asymptotic normality of the integrated periodogram in Section \ref{sec:asymptotic_normality_periodogram}. Using this result we prove in Section \ref{sec:asymptotic_normality_estimator} the main result of the paper, that is, the asymptotic normality of the estimated parameter vector. While \citet{Lii-1992} require finite moments of all orders, we only require the driving L\'evy process to have finite moments of order $4+\delta$. Furthermore, we demonstrate that the asymptotic distribution is the same in both of the following two scenarios: (i) the number of observations $n$ of the CARMA process goes to infinity and (ii) the process is observed at renewal times on the interval $[0,T]$ and $T \to \infty$. The paper concludes in Section \ref{sec:concluding_remarks} with final remarks and suggestions for future work. Several auxiliary results are presented in the Appendix.

\section{Main definitions and results} \label{sec:main_definitions_results}

We begin by recalling the definition of a CARMA process and its most relevant properties \citep{Brockwell-2024}. Let $L=(L(t))_{t \in \R}$ be a two-sided L\'evy process with mean zero and variance $\sigma_{L}^{2} \in (0,\infty)$. For $0 \leq q<p$ let $a$ and $b$ be polynomials of the form
\begin{align*}
  a(z) &= z^{p} + a_{1} z^{p-1} + \dots + a_{p-1} z + a_{p} \\
  b(z) &= b_{0} + b_{1} z + \dots + b_{q-1} z^{q-1} + z^{q}.
\end{align*}
A CARMA process $Y=(Y(t))_{t \in \R}$ is given in its \emph{state-space form} by
\begin{equation} \label{state-space_CARMA}
  Y(t) = \mathbf{b}^{\top} \mathbf{X}(t),
\end{equation}
where $\mathbf{X}=(\mathbf{X}(t))_{t \in \R}$ is a process with values in $\R^{p}$ satisfying the stochastic differential equation
\begin{equation} \label{state-space_SDE}
  d\mathbf{X}(t) = A \mathbf{X}(t) dt + \mathbf{e}_{p} dL(t),
\end{equation}
where
\begin{equation*}
  A = \begin{pmatrix}
    0 & 1 & 0 & \dots & 0 \\
    0 & 0 & 1 & \dots & 0 \\
    \vdots & \vdots & \vdots &   & \vdots \\
    0 & 0 & 0 & \dots & 1 \\
    -a_{p} & -a_{p-1} & -a_{p-2} & \dots & -a_{1}
  \end{pmatrix}, \;
  \mathbf{e}_{p} = \begin{pmatrix} 0 \\ 0 \\ \vdots \\ 0 \\ 1 \end{pmatrix}, \;
  \mathbf{b} = \begin{pmatrix} b_{0} \\ b_{1} \\ \vdots \\ b_{p-2} \\ b_{p-1} \end{pmatrix},
\end{equation*}
and $a_{1}, \dots, a_{p}, b_{0}, \dots, b_{p-1}$ are real-valued coefficients satisfying $b_{q}=1$ and $b_{j}=0$ for $j>q$. For $p=1$ we have that $A=(-a_{1})$. We assume without loss of generality (w.l.o.g.) that the polynomials $a$ and $b$ do not have any common zeros and the zeros of $a$ (or equivalently, the eigenvalues of the matrix $A$) have all strictly negative real part. This implies that the coefficients of $a$ are positive (see \citet{Hurwitz-1895}). It follows that the CARMA process $Y$ is \emph{causal} and the unique stationary solutions of \eqref{state-space_CARMA} and \eqref{state-space_SDE} are given by
\begin{equation} \label{CARMA}
  Y(t) = \int_{-\infty}^{t} \mathbf{b}^{\top} e^{A(t-u)} \mathbf{e}_{p} \, dL(u) \; \text{ and } \; \mathbf{X}(t) = \int_{-\infty}^{t} e^{A(t-u)} \mathbf{e}_{p} \, dL(u).
\end{equation}
If $L$ has finite moments of order $r$, then the moments of order $r$ of $Y$ exist and are finite (see e.g.\ \citet{Brockwell-2013}). Furthermore, $Y$ is both weakly and strongly stationary and by Proposition 3.34 of \citet{Marquardt-2007} it is strongly mixing with exponentially decaying coefficients. The autocovariance function $\gamma_{Y}$ of $Y$ is given for all $h \in \R$ by
\begin{equation} \label{autocovariance}
  \gamma_{Y}(h) = \sigma_{L}^{2} \mathbf{b}^{\top} e^{A \abs{h}} \Sigma \mathbf{b},
  \; \text{ where } \; \Sigma = \int_{0}^{\infty} e^{A s} \mathbf{e}_{p} \mathbf{e}_{p}^{\top} e^{A^{\top} s} \, ds.
\end{equation}
The spectral density $\phi_{Y}$ of $Y$ is the Fourier transform of $\gamma_{Y}$ and is given for all $u \in \R$ by
\begin{equation} \label{spectral_density}
  \phi_{Y}(u) = \frac{\sigma_{L}^{2}}{2 \pi} \biggl \lvert \frac{b(\iu u)}{a(\iu u)} \biggr \rvert^{2},
\end{equation}  
where $\iu$ is the imaginary unit. We refer to \citet{Brockwell-2001a, Brockwell-2001b, Brockwell-2014} and \citet{Brockwell-2024} for more details on CARMA processes. We are interested in estimating the parameters of the CARMA process. The sampling sequence is defined in \eqref{renewal_sampling_sequence} below using an increment process $\nu=(\nu_{k})_{k \in \Z \setminus \{0\}}$. We make the following assumptions on the L\'evy process $L$, the increment process $\nu$, and the parameter space $\Theta$. This ensures that the statements for the CARMA process $Y$ outlined above hold true for all choices of parameters in $\Theta$.
\begin{enumerate}[label=(\textbf{H\arabic*})]
\item $L=(L(t))_{t \in \R}$ is a two-sided L\'evy process with mean zero, variance $\sigma_{L}^{2} > 0$, and finite moment of order $4+\delta$ for some $\delta>0$. \label{H1}
\item $\nu=(\nu_{k})_{k \in \Z \setminus \{0\}}$ is a sequence of non-negative, independent and identically distributed (i.i.d.)\ random variables with distribution function $F$. Each $\nu_{k}$ is absolutely continuous with bounded and continuous density function $f$, has mean $1/\beta$, variance $\eta^{2}$, and a finite fourth moment. Furthermore, $\nu$ is independent of $L$. \label{H2}  
\item The parameter space $\Theta$ is a compact subset of $\R^{p+q}$. Additionally, for all $\boldsymbol{\theta} = (a_{1}, \dots, a_{p}, b_{0}, \dots, b_{q-1})^{\top} \in \Theta$ the polynomials $a$ and $b$ have no common zeros and the zeros of $a$ have all strictly negative real part. Finally, the true parameter $\boldsymbol{\theta}_{0}$ belongs to the interior of $\Theta$. \label{H3}
\end{enumerate}
We define the renewal sampling sequence $\tau=(\tau_{k})_{k \in \Z}$ by
\begin{equation} \label{renewal_sampling_sequence}
\tau_{k} \coloneqq \begin{cases}
  \sum_{j=k}^{-1} (-\nu_{j}), \qquad &k \in \Z \setminus \N_{0}, \\
  0, \qquad & k=0, \\
  \sum_{j=1}^{k} \nu_{j}, \qquad &k \in \N. \\
\end{cases}
\end{equation}
Associated with $\tau$ there is a renewal process $N=(N(t))_{t \in \R}$ given by
\begin{equation*}
  N(t) \coloneqq \begin{cases}
  \sum_{k=-\infty}^{0} (-\I_{\{ \tau_{k} > t \}}), \qquad &t<0, \\
  1, \qquad & t=0, \\
  \sum_{k=0}^{\infty} \I_{\{ \tau_{k} \leq t \}}, \qquad &t>0. \\
\end{cases}
\end{equation*}
The renewal function $R$ is defined for all $t>0$ by $R(t) \coloneqq \bE{N(t)}$. Using \ref{H2} we have that
\begin{equation*}
  R(t) = 1 + \sum_{k=1}^{\infty} F^{\ast k}(t),
\end{equation*}  
where $F^{\ast k}$ is the distribution function of $\tau_{k}$, that is, the $k$-fold convolution of $F$. Since $F$ is absolutely continuous, $N$ admits the renewal density
\begin{equation*}
  r(t) = \sum_{k=1}^{\infty} f^{\ast k}(t),
\end{equation*}
where $f^{\ast k}$ is the $k$-fold convolution of $f$. The continuity of $f$ and the renewal equation (see e.g.\ (5.29) of \citet{Alsmeyer-2012}) imply that $r$ is continuous. Furthermore, since $f$ is bounded by \ref{H2}, the renewal density theorem (see Theorem 5.26 in \citet{Alsmeyer-2012}) shows that $r$ is also bounded, which ensures that the integral below is well-defined. The sampled process $Z=(Z(t))_{t \in \R}$ is given by
\begin{equation*}
  Z(t) \coloneqq \begin{cases}
  \sum_{k=-\infty}^{0} Y(\tau_{k}) \I_{\{ \tau_{k} > t \}}, \qquad &t<0, \\
  Y(0), \qquad & t=0, \\
  \sum_{k=0}^{\infty} Y(\tau_{k}) \I_{\{ \tau_{k} \leq t \}}, \qquad &t>0. \\
\end{cases}
\end{equation*}
The reduced covariance measure $\mu_{N}$ of the renewal process $N$ is defined (up to the factor $\beta$) as in \citet{Lii-1992} by
\begin{equation*}
  \mu_{N}(B) = \delta_{0}(B) + \int_{B} r(\abs{h}) \, dh
\end{equation*}
for all Borel sets $B$. As in \citet{Lii-1992}, the spectral density $\phi_{Z}$ of $Z$ is given for all $u \in \R$ by
\begin{align}
  \phi_{Z}(u) &= \frac{1}{2 \pi} \int_{-\infty}^{\infty} e^{-\iu hu} \gamma_{Y}(h) \, d\mu_{N}(h) \nonumber \\
  &= \frac{1}{2 \pi} \biggl( \sigma_{L}^{2} \mathbf{b}^{\top} \Sigma \mathbf{b} + \int_{-\infty}^{\infty} e^{-\iu hu} \gamma_{Y}(h) r(\abs{h}) \, dh \biggr). \label{spectral_density_Z}
\end{align}
If $\nu_{k}$ are exponentially distributed, then $r(\cdot) \equiv \beta$ and, using \eqref{spectral_density}, the spectral density reduces to
\begin{equation} \label{spectral_density_Z_exponential_sampling}
  \phi_{Z}(u) = \frac{\beta \sigma_{L}^{2}}{2 \pi} \biggl( \frac{\mathbf{b}^{\top} \Sigma \mathbf{b}}{\beta} + \biggl \lvert \frac{b(\iu u)}{a(\iu u)} \biggr \rvert^{2} \biggr).
\end{equation}
We also write $\phi_{Y}(u,\boldsymbol{\theta})$ and $\phi_{Z}(u,\boldsymbol{\theta})$ for $\phi_{Y}(u)$ and $\phi_{Z}(u)$ to highlight the dependence on the parameter $\boldsymbol{\theta}$. We recall that the true parameter is denoted by $\boldsymbol{\theta}_{0}$. We define a Whittle-type estimator following \citet{Lii-1992}. First, we let
\begin{equation} \label{function_g}
  g(u, \boldsymbol{\theta}) = \frac{\phi_{Z}(u, \boldsymbol{\theta})}{s^{2}(\boldsymbol{\theta})}, \, \text{ where } \, s^{2}(\boldsymbol{\theta}) = \int_{-\infty}^{\infty} \frac{\phi_{Z}(u, \boldsymbol{\theta})}{1+u^{2}} \, du.
\end{equation}
Then, we define $K: \R^{p+q} \to \R$ by
\begin{equation*}
  K(\boldsymbol{\theta}) = \int_{-\infty}^{\infty} \frac{\log(g(u,\boldsymbol{\theta}))}{1+u^{2}} \phi_{Z}(u, \boldsymbol{\theta}_{0}) \, du.
\end{equation*}
In general, we need the following anti-aliasing assumption (see \citet{Lii-1992}, Assumption 2.1).
\begin{enumerate}[label=(\textbf{H\arabic*})]
\setcounter{enumi}{3}  
\item For all $\boldsymbol{\theta}_{1}, \boldsymbol{\theta}_{2} \in \Theta$ with $\boldsymbol{\theta}_{1} \neq \boldsymbol{\theta}_{2}$ the functions $g(\cdot, \boldsymbol{\theta}_{1})$ and $g(\cdot, \boldsymbol{\theta}_{2})$ differ on a set of positive (Lebesgue) measure. \label{H4}  
\end{enumerate}
We notice that \ref{H4} holds true when $\nu_{k}$ are exponentially distributed (see \citet{Lii-1992} pp. 61-62 and recall \eqref{spectral_density}).
\begin{proposition}[\citet{Lii-1992}] \label{proposition:function_K}
Assume \ref{H1}-\ref{H4}. The function $K$ is well-defined, continuous, and assumes a unique global maximum value at $\boldsymbol{\theta}_{0} \in \Theta$.
\end{proposition}
Since the spectral density $\phi_{Z}$ is generally unknown, we estimate it using the periodogram $I_{Z,N(T)}$, where for all $n \in \N$
\begin{equation} \label{periodogram}
  I_{Z,n}(u) = \frac{1}{2 \pi n} \biggl \lvert \sum_{k=1}^{n} e^{-\iu u \tau_{k}} Y(\tau_{k}) \biggr \rvert^{2}.
\end{equation}  
This leads to the estimator $\hat{K}_{n}$ for $K$ given by
\begin{equation*}
  \hat{K}_{n}(\boldsymbol{\theta}) = \int_{-\infty}^{\infty} \frac{\log(g(u,\boldsymbol{\theta}))}{1+u^{2}} I_{Z,n}(u) \, du.
\end{equation*}
An estimator $\hat{\boldsymbol{\theta}}_{n}$ of $\boldsymbol{\theta}_{0}$ is any point in $\Theta$ maximizing $\hat{K}_{n}$. Similarly, the estimator $\hat{\boldsymbol{\theta}}_{N(T)}$ is a maximizer of $\hat{K}_{N(T)}$.
We establish asymptotic normality of both $(\hat{\boldsymbol{\theta}}_{n})_{n \in \mathbb{N}}$ and $(\hat{\boldsymbol{\theta}}_{N(T)})_{T \geq 0}$. The covariance matrix is given by $\Sigma_{0} \coloneqq W^{-1}QW^{-1}$, where the entries of $W$ are the second order partial derivative of $K(\cdot)$ (see \eqref{entries_W} below for more details) and the entries of $Q$ are given by \eqref{entries_Q} below.

\begin{theorem} \label{theorem:asymptotic_normality}
  Assume \ref{H1}-\ref{H4} and suppose that the matrix $W$ is invertible. Then
\begin{align*}
  \text{(i) } &\sqrt{n} (\hat{\boldsymbol{\theta}}_{n} - \boldsymbol{\theta}_{0}) \xrightarrow[n \to \infty]{d} \mathcal{N}(\boldsymbol{0}, \Sigma_{0}) \text{ and} \\
  \text{(ii) } &\sqrt{N(T)} (\hat{\boldsymbol{\theta}}_{N(T)} - \boldsymbol{\theta}_{0}) \xrightarrow[T \to \infty]{d} \mathcal{N}(\boldsymbol{0}, \Sigma_{0}).
\end{align*}
\end{theorem}
\noindent An immediate consequence of Theorem \ref{theorem:asymptotic_normality} is the consistency of the estimators $(\hat{\boldsymbol{\theta}}_{n})_{n \in \mathbb{N}}$ and $(\hat{\boldsymbol{\theta}}_{N(T)})_{T \geq 0}$.
\begin{corollary} \label{corollary:asymptotic_normality}
  Assume \ref{H1}-\ref{H4}. Then, $\hat{\boldsymbol{\theta}}_{n} \xrightarrow[n \to \infty]{p} \boldsymbol{\theta}_{0}$ and $\hat{\boldsymbol{\theta}}_{N(T)}  \xrightarrow[T \to \infty]{p} \boldsymbol{\theta}_{0}$.
\end{corollary}  
\noindent The proof of Theorem \ref{theorem:asymptotic_normality} is divided into several steps. We begin by showing in Section \ref{sec:asymptotic_normality_periodogram} below that the integrated periodogram is asymptotically normal. The proof of Theorem \ref{theorem:asymptotic_normality} is contained in Section \ref{sec:asymptotic_normality_estimator}. We critically use that for $\boldsymbol{\theta}_{0} \in \Theta$ the CARMA process $Y$ is strongly mixing with exponentially decaying coefficients. It follows from Proposition 4.1 of \citet{Brandes-2023} that both $(Y(\tau_{k}), \tau_{k}-\tau_{k-1})_{k \in \Z}$ and $(Y(\tau_{k}))_{k \in \Z}$ are strongly mixing with exponentially decaying coefficients. In particular, they are ergodic.

Another important aspect of our analysis is the estimation of the variance $\sigma_{L}^{2}$ of the driving L\'evy process. To this end, we follow the approach outlined on p.\ 68 by \citet{Lii-1992}. We begin by considering the estimator
\begin{equation*}
  \hat{s}_{n}^{2}(\boldsymbol{\theta}_{0}) = \int_{-\infty}^{\infty} \frac{I_{Z,n}(u)}{1+u^{2}} \, du
\end{equation*}
for $s^{2}(\boldsymbol{\theta}_{0})$ in \eqref{function_g}. Using the notation $s^{2}(\boldsymbol{\theta}_{0}, \sigma_{L}^{2})$ for $s^{2}(\boldsymbol{\theta}_{0})$ to highlight the dependence on $\sigma_{L}^{2}$, we see that $s^{2}(\boldsymbol{\theta}_{0}, \sigma_{L}^{2}) = \sigma_{L}^{2} \tilde{s}^{2}(\boldsymbol{\theta}_{0})$, where
\begin{equation*}
  \tilde{s}^{2}(\boldsymbol{\theta}_{0}) = \int_{-\infty}^{\infty} \frac{\tilde{\phi}_{Z}(u, \boldsymbol{\theta}_{0})}{1+u^{2}} \, du \text{ and } \tilde{\phi}_{Z}(u, \boldsymbol{\theta}_{0}) = \frac{\phi_{Z}(u, \boldsymbol{\theta}_{0})}{\sigma_{L}^{2}}
\end{equation*}
no longer depend on $\sigma_{L}^{2}$. Replacing $\boldsymbol{\theta}_{0}$ by $\hat{\boldsymbol{\theta}}_{n}$ we obtain the estimator
\begin{equation*}
  \hat{\sigma}_{L,n}^{2} = \frac{\hat{s}_{n}^{2}(\boldsymbol{\theta}_{0})}{\tilde{s}^{2}(\hat{\boldsymbol{\theta}}_{n})}.
\end{equation*}
for the variance of $L$. Since $\tilde{s}^{2}(\cdot)$ is positive and continuous (see Appendix \ref{sec:equicontinuity} for more details), Corollary \ref{corollary:asymptotic_normality} yields that $\tilde{s}^{2}(\hat{\boldsymbol{\theta}}_{n}) \xrightarrow[n \to \infty]{p} \tilde{s}^{2}(\boldsymbol{\theta}_{0})$. Using this and \eqref{convergence_s2} below, we deduce that
\begin{equation*}
  \hat{\sigma}_{L,n}^{2} \xrightarrow[n \to \infty]{p} \sigma_{L}^{2}
\end{equation*}
establishing the consistency of the variance estimator. Exactly the same result holds when $n$ is replaced by $N(t)$ and $t \to \infty$.

We conclude this section by discussing some special cases and simulation studies. In particular, the Ornstein-Uhlenbeck process is obtained as a particular case of the CARMA process by taking $p=1$ and $q=0$. Using \eqref{CARMA} with $a(z)=z+a_{1}$ and $b(z)=b_{0}=1$, we obtain that
\begin{equation*}
  Y(t) = \int_{-\infty}^{t} e^{-a_{1}(t-u)} \, dL(u).
\end{equation*}
In this case the parameter is $\theta=a_{1}>0$ and we denote by $\theta_{0}$ the true parameter. We see from \eqref{autocovariance} and \eqref{spectral_density} that $Y$ has autocovariance function
\begin{equation*}
  \gamma_{Y}(h) = \frac{\sigma_{L}^{2} e^{-\theta \abs{h}}}{2\theta}
\end{equation*}
and spectral density
\begin{equation*}
  \phi_{Y}(u,\theta) = \frac{\sigma_{L}^{2}}{2 \pi} \, \frac{1}{\theta^{2}+u^{2}},
\end{equation*}
where we have used that $\Sigma=(2\theta)^{-1}$. If $\nu_{k}$ are exponentially distributed with mean $(\beta)^{-1}$, using \eqref{spectral_density_Z_exponential_sampling}, we see that the sampled process $Z$ has spectral density
\begin{equation*}
  \phi_{Z}(u,\theta) = \frac{\beta \sigma_{L}^{2}}{2 \pi} \biggl( \frac{1}{2\theta \beta} + \frac{1}{\theta^{2}+u^{2}} \biggr).
\end{equation*}
Finally, using \eqref{function_g} and an integral calculation, we have
\begin{equation*}
  g(u,\theta) = \biggl( \frac{1}{2\theta \beta} + \frac{1}{\theta^{2}+u^{2}} \biggr) \biggl/ \biggl( \frac{\pi}{2 \theta \beta} + \frac{\pi}{\theta(\theta+1)} \biggr).
\end{equation*}
As before
\begin{equation*}
  K(\theta) = \int_{-\infty}^{\infty} \frac{\log(g(u,\theta))}{1+u^{2}} \phi_{Z}(u,\theta_{0}) \, du
\end{equation*}
is estimated by
\begin{equation*}
  \hat{K}_{n}(\theta) = \int_{-\infty}^{\infty} \frac{\log(g(u,\theta))}{1+u^{2}} I_{Z,n}(u) \, du.
\end{equation*}
We simulate the Ornstein-Uhlenbeck process assuming the driving L\'evy process is (i) a standard Brownian motion or (ii) a Gamma process with shape parameter $0.2$ and rate parameter $0.3$ as in Section 14.2 of \citet{Brockwell-2024} and centered so that the mean is zero. To simulate the process we use the R-programs in Section 19.5 of \citet{Brockwell-2024} with discretization step size $h=10^{-3}$. The above integrals are approximated using a trapezoidal rule and the sample size is set to $n \in \{ 100, 1000 \}$. The true parameter value is taken to be $\theta_{0}=1$ and the inter-arrival times are assumed to be exponentially distributed with parameter $\beta \in \{ 0.5, 1, 2, 5 \}$. Finally, the sampled process is obtained by linear interpolation of the discretized process. The procedure is repeat $100$ times and the mean and variance of the estimated parameter $\hat{\theta}_{n}$ are summarized in Tables \ref{table_1} and \ref{table_2} for the Brownian motion and Gamma process driven models, respectively. The simulations clearly show that bias and the standard deviations decrease for increasing sample size.

\begin{table}[h]
\centering
\begin{tabular}{|c|c|c|c|c|}
\hline
  & $n=100$ & $n=1000$ \\ \hline
$\beta=0.5$ & $1.03$ $(0.18)$ & $0.89$ $(0.09)$ \\ \hline
$\beta=1$ & $1.08$ $(0.24)$ & $0.95$ $(0.05)$ \\ \hline
$\beta=2$ & $1.18$ $(0.54)$ & $0.99$ $(0.04)$ \\ \hline
$\beta=5$ & $1.45$ $(1.11)$ & $1.04$ $(0.09)$ \\ \hline
\end{tabular}
\caption{Mean and variance (in parenthesis) of the estimated parameter $\hat{\theta}_{n}$ for a Brownian motion driven Ornstein-Uhlenbeck process over $100$ repetitions.}
\label{table_1}
\end{table}

\begin{table}[h]
\centering
\begin{tabular}{|c|c|c|c|c|}
\hline
  & $n=100$ & $n=1000$ \\ \hline
$\beta=0.5$ & $1.08$ $(0.30)$ & $0.97$ $(0.06)$ \\ \hline
$\beta=1$ & $1.06$ $(0.21)$ & $0.96$ $(0.06)$ \\ \hline
$\beta=2$ & $1.21$ $(0.51)$ & $1.01$ $(0.05)$ \\ \hline
$\beta=5$ & $1.20$ $(0.58)$ & $1.04$ $(0.07)$ \\ \hline
\end{tabular}
\caption{Mean and variance (in parenthesis) of the estimated parameter $\hat{\theta}_{n}$ for a Gamma process driven Ornstein-Uhlenbeck process over $100$ repetitions.}
\label{table_2}
\end{table}

\section{Asymptotic normality of the integrated periodogram} \label{sec:asymptotic_normality_periodogram}

For a function $G \in L^{1}(\R) \cap L^{2}(\R)$, we let
\begin{equation*}
  J_{n} \coloneqq \int_{-\infty}^{\infty} G(u) I_{Z,n}(u) \, du \; \text{ and } \; J \coloneqq \int_{-\infty}^{\infty} G(u) \phi_{Z}(u, \boldsymbol{\theta}_{0}) \, du.
\end{equation*}
We show in Lemma \ref{lemma:clt} below that $J_{n}$ is asymptotically normal with mean $J$. We begin with several preliminary results. We denote by 
\begin{equation*}
  \hat{G}_{R}(\xi) \coloneqq \frac{1}{2\pi} \int_{-\infty}^{\infty} G(u) \cos(\xi u) \, du
\end{equation*}
the real part of the Fourier transform of $G(\cdot)$. We notice that $\hat{G}_{R}$ is in $L^{2}(\mathbb{R})$, uniformly continuous, and bounded in absolute value by $\norm{G}_{L^{1}}/(2 \pi)$, where $\norm{G}_{L^{1}}$ is the $L^{1}$-norm of $G$. Using \eqref{periodogram} one can show that
\begin{align*}
  J_{n} &= \frac{1}{n} \sum_{k=1}^{n} \sum_{j=1}^{n} \hat{G}_{R}(\tau_{k}-\tau_{j}) Y(\tau_{k}) Y(\tau_{j}) \\
  &= \frac{1}{n} \sum_{j=1}^{n} \sum_{k=-(n-j)}^{n-j} U_{j}(\abs{k}),
\end{align*}
where $U_{j}(k) \coloneqq \hat{G}_{R}(\tau_{j+k}-\tau_{j}) Y(\tau_{j}) Y(\tau_{j+k})$ for $j,k \in \N_{0}$. We also let
\begin{equation*}
  H_{j}(n) \coloneqq \sum_{k=-n}^{n} U_{j}(\abs{k}).
\end{equation*}  

\begin{lemma} \label{lemma:mixing}
  Assume \ref{H1}-\ref{H3}. For every $k,n \in \N_{0}$ the sequences $(U_{j}(k))_{j \in \N_{0}}$ and $(H_{j}(n))_{j \in \N_{0}}$ are strongly stationary and strongly mixing with exponentially decaying coefficients.
\end{lemma}
\begin{proof}
  Using Proposition 2.1 of \citet{Brandes-2019} we obtain that $(Y(\tau_{l}), \tau_{l}-\tau_{l-1})_{l \in \Z}$ is strongly stationary. Since $(U_{j}(k))_{j \in \N_{0}}$ is a measurable function of a strongly stationary sequence, it is also strongly stationary. Next, using Proposition 3.34 of \citet{Marquardt-2007} and Proposition 4.1 of \citet{Brandes-2023}, we see that $(Y(\tau_{l}), \tau_{l}-\tau_{l-1})_{l \in \Z}$ is strongly mixing with exponentially decaying coefficients $\alpha_{Y,\tau}(h)=O(e^{-c h})$, where $c>0$ and $h \in \N$. Since $U_{j}(k)$ is a measurable function of $(Y(\tau_{l}), \tau_{l}-\tau_{l-1})_{l \in \{ j, \dots, j+k \} }$, whereas $U_{j+h}(k)$ is a measurable function of $(Y(\tau_{l}), \tau_{l}-\tau_{l-1})_{l \in \{ j+h, \dots, j+h+k \} }$, we deduce that $(U_{j}(k))_{j \in \N_{0}}$ is strongly mixing with coefficients $(\alpha_{U}(h))_{h \in \N}$, where
\begin{equation*}
  \alpha_{U}(h) \leq \begin{cases}
    1 &\text{ if } 1 \leq h \leq k \\
    \alpha_{Y,\tau}(h-k) &\text{ if } h > k.
\end{cases}    
\end{equation*}
We conclude that $\alpha_{U}(h)=O(e^{-c h})$. The proof for $(H_{j}(n))_{j \in \N_{0}}$ is similar.
\end{proof}
Next, we establish an upper bound on the autocovariance function of $Y$. Clearly, both the matrix $A$ and $\gamma_{Y}(\cdot)$ depend on the parameter $\boldsymbol{\theta}_{0} \in \Theta$, that is, $A=A(\boldsymbol{\theta}_{0})$ and $\gamma_{Y}(\cdot) = \gamma_{Y}(\cdot, \boldsymbol{\theta}_{0})$. We provide a bound that holds uniformly for all parameters.
\begin{lemma} \label{lemma:autocovariance}
Assume \ref{H1}-\ref{H3}. Then, there exist $C,D >0$ such that $\abs{\gamma_{Y}(h, \boldsymbol{\theta})} \leq C e^{-D \abs{h}}$ for all $h \in \R$ and $\boldsymbol{\theta} \in \Theta$.
\end{lemma}
\begin{proof}
  We denote by $\norm{\cdot}$ the Euclidean norm as well as the corresponding operator norm, that is, the spectral norm. Using \eqref{exponential_decrease_autocovariance} in Appendix \ref{sec:uniform_upper_bounds} we obtain that
\begin{equation*}
    \abs{\gamma_{Y}(h, \boldsymbol{\theta})} \leq \doubletilde{C} \sigma_{L}^{2} e^{-D \abs{h}}
\end{equation*}
for a positive constant $\doubletilde{C}>0$ and all $\boldsymbol{\theta} \in \Theta$. The result follows by taking $C=\doubletilde{C} \sigma_{L}^{2}$.
\end{proof}
We define the truncated reduced covariance measure $\bar{\mu}_{N,n}$ for all Borel sets $B$ by
\begin{equation*}
  \bar{\mu}_{N,n}(B) = \delta_{0}(B) + \int_{B} \sum_{k=1}^{n} f^{\ast k}(\abs{h}) \, dh
\end{equation*}
and the truncated spectral density $\bar{\phi}_{Z,n}(\cdot, \boldsymbol{\theta}_{0})$ for all $u \in \R$ by
\begin{equation} \label{truncated_spectral_density}
  \bar{\phi}_{Z,n}(u, \boldsymbol{\theta}_{0}) = \frac{1}{2 \pi} \int_{-\infty}^{\infty} e^{-\iu hu} \gamma_{Y}(h) \, d \bar{\mu}_{N,n}(h).
\end{equation}
Finally, let
\begin{equation*}
  \bar{J}_{n} \coloneqq \int_{-\infty}^{\infty} G(u) \bar{\phi}_{Z,n}(u, \boldsymbol{\theta}_{0}) \, du.
\end{equation*}
We show below that $(H_{j}(n))_{j \in \N_{0}}$ has expectation $\bar{J}_{n}$, which converges exponentially fast to $J$ (for the upcoming central limit theorem the rate $\sqrt{n}$ suffices).

\begin{lemma} \label{lemma:mean_variance}
  Assume \ref{H1}-\ref{H3}. Then, $\bE{H_{0}(n)} = \bar{J}_{n}$ and $\lim_{n \to \infty} \sqrt{n} (J - \bar{J}_{n}) = 0$.
\end{lemma}
\begin{proof}
By conditioning on $\tau$ and taking expectation with respect to $Y$ we obtain that
\begin{align*}
  \bE{H_{0}(n)} &= \frac{1}{2\pi} \mathbf{E} \biggl[ \sum_{k=-n}^{n} \biggl[ \int_{-\infty}^{\infty} G(u) e^{-\iu u \tau_{\abs{k}}} \, du  \biggr] \gamma_{Y}(\tau_{\abs{k}}) \biggr] \\
  &= \frac{1}{2\pi} \int_{-\infty}^{\infty} G(u) \mathbf{E} \biggl[ \sum_{k=-n}^{n} e^{-\iu u \tau_{\abs{k}}} \gamma_{Y}(\tau_{\abs{k}}) \biggr] \, du.
\end{align*}
Next, we compute 
\begin{align*}
  \mathbf{E} \biggl[ \sum_{k=-n}^{n} e^{-\iu u \tau_{\abs{k}}} \gamma_{Y}(\tau_{\abs{k}}) \biggr] &= \gamma_{Y}(0) + \int_{-\infty}^{\infty} e^{-\iu u h} \gamma_{Y}(h) \sum_{k=1}^{n} f^{\ast k}(\abs{h}) \, dh \\
  &= \int_{-\infty}^{\infty} e^{-\iu u h} \gamma_{Y}(h) \, d \bar{\mu}_{N,n}(h).
\end{align*}
We conclude that
\begin{equation*}
  \bE{H_{0}(n)} = \frac{1}{2\pi} \int_{-\infty}^{\infty} G(u) \biggl[ \int_{-\infty}^{\infty} e^{-\iu u h} \gamma_{Y}(h) \, d \bar{\mu}_{N,n}(h) \biggr] \, du = \bar{J}_{n}.
\end{equation*}
Next, let
\begin{equation} \label{Laplace_transform}
  \mathcal{L}(\xi) \coloneqq \int_{0}^{\infty} e^{-\xi h} f(h) \, dh
\end{equation}  
be the Laplace transform of $f$. Using that $G \in L^{1}(\R)$ and Lemma \ref{lemma:autocovariance}, we find positive constants $C$ and $D$ such that
\begin{align*}
  2 \pi \sqrt{n} \abs{J-\bar{J}_{n}} &\leq \sqrt{n} \int_{-\infty}^{\infty} \abs{G(u)} \biggl( \int_{-\infty}^{\infty} \abs{e^{-\iu u h}} \abs{\gamma_{Y}(h)} \sum_{k=n+1}^{\infty} f^{\ast k}(\abs{h}) \, dh \biggr) \, du \\
   &\leq \norm{G}_{L^{1}} \sqrt{n} \int_{-\infty}^{\infty} \abs{\gamma_{Y}(h)} \sum_{k=n+1}^{\infty} f^{\ast k}(\abs{h}) \, dh \\
   &\leq 2 C \norm{G}_{L^{1}} \sqrt{n} \sum_{k=n+1}^{\infty} \int_{0}^{\infty} e^{-D h} f^{\ast k}(h) \, dh \\
  &\leq 2 C \norm{G}_{L^{1}} \sqrt{n} \sum_{k=n+1}^{\infty} ( \mathcal{L}(D) )^{k}.
\end{align*}
Since $\mathcal{L}(D) < 1$ as $D>0$, we conclude that
\begin{equation*}
  \sum_{k=n+1}^{\infty} ( \mathcal{L}(D) )^{k} = \frac{( \mathcal{L}(D) )^{n+1}}{1-\mathcal{L}(D)}
\end{equation*}
converges exponentially fast to zero as $n \to \infty$.
\end{proof}
\noindent For all $j \in \mathbb{N}_{0}$ and $m \in \mathbb{N}$ we define
\begin{equation} \label{Sm}
  S_{j}(m) \coloneqq H_{j}(m) - \bar{J}_{m} = \bar{U}_{j}(0) + 2 \sum_{k=1}^{m} \bar{U}_{j}(k),
\end{equation}
where $\bar{U}_{j}(k) \coloneqq U_{j}(k) - \bE{U_{j}(k)}$ for all $j,k \in \mathbb{N}_{0}$. We compute mean and covariance of $(U_{j}(k))_{j \in \N_{0}}$ in Appendix \ref{sec:explicit_formulas}. The next lemma ensures that variances and covariances exist and are finite under assumption \ref{H1}.

\begin{lemma} \label{lemma:moments_of_U}
  Suppose that $L$ has finite moments of order $2 r$, where $r \geq 1$. Then, the moments $\bE{\abs{U_{j}(k)}^{r}}$, $\bE{\abs{\bar{U}_{j}(k)}^{r}}$, and $\bE{\abs{S_{j}(m)}^{r}}$ are finite for all $j,k \in \mathbb{N}_{0}$ and $m \in \mathbb{N}$.
\end{lemma}  
\begin{proof}
Since $\hat{G}_{R}$ is bounded in absolute value by $\norm{G}_{L^{1}}/(2\pi)$ we have
\begin{align*}
  \bE{\abs{U_{j}(k)}^{r}} &= \bE{ \abs{\hat{G}_{R}(\tau_{j+k}-\tau_{j})}^{r}\abs{Y(\tau_{k}) Y(\tau_{j+k})}^{r} } \\
  &\leq \biggl( \frac{\norm{G}_{L^{1}}}{2\pi} \biggr)^{r} \bE{ \abs{Y(\tau_{k}) Y(\tau_{j+k})}^{r} }.
\end{align*}
By conditioning on $\tau$ and using Cauchy–Schwarz inequality, we obtain that the expectation on the right-hand side (RHS) is bounded by
\begin{equation*}
 \bE{ \, (\bE{\abs{Y(\tau_{k})}^{2r} \, | \, \tau \,})^{1/2} (\bE{\abs{Y(\tau_{j+k})}^{2r} \, | \, \tau \,})^{1/2} }.
\end{equation*}
By stationarity the conditional moments above are both equal to $\bE{\abs{Y(0)}^{2r}}$. We conclude that
\begin{equation*}
  \bE{\abs{U_{j}(k)}^{r}} \leq \biggl( \frac{\norm{G}_{L^{1}}}{2\pi} \biggr)^{r} \, \bE{\abs{Y(0)}^{2r}},
\end{equation*}
where the RHS is finite since $G \in L^{1}(\mathbb{R})$ and $L$ has finite moments of order $2r$ (see \citet{Brockwell-2013}). It is easy to see that $\bE{\abs{U_{j}(k)}^{r}} < \infty$ if and only if $\bE{\abs{\bar{U}_{j}(k)}^{r}} < \infty$. Since $\abs{\, \cdot \,}^{r}$ is a convex function, we have
\begin{equation*}
  \biggl| \frac{1}{m+1} \sum_{j=0}^{m} a_{j} \biggr|^{r} \leq \frac{1}{m+1} \sum_{j=0}^{m} \abs{a_{j}}^{r}
\end{equation*}
for all $a_{0}, \dots, a_{m} \in \mathbb{R}$. Using this inequality and \eqref{Sm}, we obtain that
\begin{equation*}
  \bE{\abs{S_{j}(m)}^{r}} \leq (m+1)^{r-1} \biggl( \bE{\abs{\bar{U}_{j}(0)}^{r}} + 2^{r} \sum_{k=1}^{m} \bE{\abs{\bar{U}_{j}(k)}^{r}} \biggr).
\end{equation*}
Since $\bE{\abs{\bar{U}_{j}(k)}^{r}}<\infty$, the above RHS is finite.
\end{proof}
Our next result gives asymptotic normality of $J_{n}$ and $J_{N(T)}$. The mean and covariance of $(U_{j}(k))_{j \in \N_{0}}$, computed in Appendix \ref{sec:moments_and_variance}, can be used to derive explicit formulas for the variances $\sigma_{J,m}^{2}$ and $\sigma_{J}^{2}$ defined in \eqref{variance_of_Sm} and \eqref{limiting_variance} below.

\begin{lemma} \label{lemma:clt}
  Assume \ref{H1}-\ref{H3}. Then
\begin{align*}
  \text{(i) } &\sqrt{n} (J_{n} - J) \xrightarrow[n \to \infty]{d} \mathcal{N}(0, \sigma_{J}^{2}) \text{ and} \\
  \text{(ii) } &\sqrt{N(T)} (J_{N(T)} - J) \xrightarrow[T \to \infty]{d} \mathcal{N}(0, \sigma_{J}^{2}).
\end{align*}
\end{lemma}
\begin{proof}
  We begin by proving (i). By Lemma \ref{lemma:mean_variance} it is enough to show that $\mathcal{S}_{n} \xrightarrow[n \to \infty]{d} \mathcal{N}(0, \sigma_{J}^{2})$, where
\begin{equation*}
  \mathcal{S}_{n} \coloneqq \frac{1}{\sqrt{n}} \sum_{j=1}^{n} S_{j}(n-j).
\end{equation*}
For $m \in \N$ we consider the sequence $S(m)=(S_{j}(m))_{j \in \N_{0}}$. By Lemma \ref{lemma:mixing} it is strongly stationary and strongly mixing with exponentially decaying coefficients $(\alpha_{S(m)}(h))_{h \in \N}$. We have that
\begin{align*}
  \mathbf{Var} \biggl[ \sum_{j=1}^{n} S_{j}(m) \biggr] &= \sum_{i=1}^{n} \bE{S_{i}^{2}(m)} +  2 \sum_{i=1}^{n} \sum_{j=1}^{n-i} \bE{S_{i}(m) S_{i+j}(m)} \\
  &= n \bE{S_{0}^{2}(m)} + 2 \sum_{i=1}^{n} \sum_{j=1}^{n-i} \bE{S_{0}(m) S_{j}(m)}
\end{align*}
yielding that
\begin{equation*}
  \lim_{n \to \infty} \mathbf{Var} \biggl[ \frac{1}{\sqrt{n}} \sum_{j=1}^{n} S_{j}(m) \biggr] = \sigma_{J,m}^{2},
\end{equation*}
where
\begin{equation} \label{variance_of_Sm}
  \sigma_{J,m}^{2} \coloneqq \bE{S_{0}^{2}(m)} + 2 \sum_{j=1}^{\infty} \bE{S_{0}(m) S_{j}(m)}.
\end{equation}
The strong mixing property of $S(m)$ ensures that $\sigma_{J,m}^{2}$ is well-defined and finite (see below for more details). Let $\mathcal{\tilde{S}}_{n,m} \coloneqq \frac{1}{\sqrt{n}} \sum_{j=1}^{n} S_{j}(m)$. Then, by Proposition 6.3.9 of \citet{Brockwell-1991} it suffices to show that
\begin{align} \label{CLT_for_Sm}
  &\mathcal{\tilde{S}}_{n,m} \xrightarrow[n \to \infty]{d} \mathcal{N}(0, \sigma_{J,m}^{2}), \\ \label{limiting_variance}
  &\sigma_{J}^{2} \coloneqq \lim_{m \to \infty} \sigma_{J,m}^{2} \text{ exists and is finite, and } \\ \label{continuity_m}
  &\lim_{m \to \infty} \limsup_{n \to \infty} \bP{\abs{\mathcal{S}_{n} - \mathcal{\tilde{S}}_{n,m}} \geq \epsilon } = 0 \text{ for all } \epsilon>0.
\end{align}
We first notice that $\sigma_{J,m}^{2} < \infty$. To see this, we use an inequality due to \citet{Davydov-1968}: for any random variables $X_{1}$ and $X_{2}$ with mean zero and positive real numbers $p,q,r$ with $1/p+1/q+1/r=1$
\begin{equation*}
  \abs{\bE{X_{1}X_{2}}} \leq 12 (\bE{\abs{X_{1}}^{p}})^{1/p} (\bE{\abs{X_{2}}^{q}})^{1/q} (\alpha_{X})^{1/r},
\end{equation*}
where $\alpha_{X}$ is the strong mixing coefficient of the $\sigma$-algebras generated by $X_{1}$ and $X_{2}$. Using this with $p=q=2+\delta/2$ and $r=(4+\delta)/\delta$ we obtain
\begin{equation} \label{finite_variance}
  \bE{S_{0}(m) S_{j}(m)} \leq 12 (\bE{\abs{S_{0}(m)}^{2+\delta/2}})^{4/(4+\delta)} \alpha_{S(m)}(j)^{\delta/(4+\delta)}.
\end{equation}
Using Lemma \ref{lemma:moments_of_U} with $r=2+\delta/2$ and \ref{H1}, we have
\begin{equation} \label{moment_Sm}
  \bE{\abs{S_{0}(m)}^{2+\delta/2}} < \infty.
\end{equation}
Since the $\alpha$-mixing coefficients of $S(m)$ are decreasing exponentially fast, we also obtain
\begin{equation} \label{mixing_coefficients_Sm}
  \sum_{j=1}^{\infty} \alpha_{S(m)}^{\delta/(4+\delta)}(j) < \infty.
\end{equation}
\eqref{finite_variance}, \eqref{moment_Sm} and \eqref{mixing_coefficients_Sm} imply that the RHS of \eqref{variance_of_Sm} is finite. Using \eqref{moment_Sm} and \eqref{mixing_coefficients_Sm}, we apply Theorem 18.5.3 of \citet{Ibragimov-1971} and obtain \eqref{CLT_for_Sm}. We now show \eqref{limiting_variance}. Using \eqref{Sm} we see that
\begin{align*}
  \bE{S_{0}(m) S_{j}(m)} &= \bE{\bar{U}_{0}(0) \bar{U}_{j}(0)} + 2 \sum_{k=1}^{m} \bE{\bar{U}_{0}(0) \bar{U}_{j}(k)} \\
  &+ 2 \sum_{l=1}^{m} \bE{\bar{U}_{0}(l) \bar{U}_{j}(0)} + 4 \sum_{k=1}^{m} \sum_{l=1}^{m} \bE{\bar{U}_{0}(l) \bar{U}_{j}(k)}
\end{align*}  
for all $j \in \mathbb{N}_{0}$. Using this in \eqref{variance_of_Sm}, we obtain that
\begin{equation} \label{sigma_m}
\begin{split}  
  \sigma_{J,m}^{2} &= \bE{\bar{U}_{0}^{2}(0)} + 4 \sum_{l=0}^{m} \sum_{k=1}^{m}\bE{\bar{U}_{0}(l) \bar{U}_{0}(k)} + 2 \sum_{j=1}^{\infty} \biggl( \bE{ \bar{U}_{0}(0) \bar{U}_{j}(0) } \\
  &+ 2 \sum_{k=1}^{m} \bE{\bar{U}_{0}(0) \bar{U}_{j}(k) } + 2 \sum_{l=1}^{m} \bE{ \bar{U}_{0}(l) \bar{U}_{j}(0) } + 4 \sum_{k=1}^{m} \sum_{l=1}^{m} \bE{ \bar{U}_{0}(l) \bar{U}_{j}(k) } \biggr).
\end{split}  
\end{equation}
Thus, we expect that $\sigma_{J}^{2} = \lim_{m \to \infty} \sigma_{J,m}^{2} = \sigma_{J,\infty}^{2}$ is given by the above expression with $m$ replaced by $\infty$. We make this precise in Appendix \ref{sec:finiteness_of_sigma2} by showing that the above series (with $m=\infty$) are absolutely convergent.

We now turn to \eqref{continuity_m}. We have
\begin{equation*}
  \mathcal{S}_{n} - \mathcal{\tilde{S}}_{n,m} = \frac{1}{\sqrt{n}} \sum_{j=1}^{n-m} (S_{j}(n-j) - S_{j}(m)) + \frac{1}{\sqrt{n}} \sum_{j=n-m+1}^{n} (S_{j}(n-j) - S_{j}(m))
\end{equation*} 
and by Chebyshev's inequality it is enough to show that both
\begin{align} \label{continuity_m1}
  \lim_{m \to \infty} \limsup_{n \to \infty} \frac{1}{n} \mathbf{E} \biggl[ \biggl( \sum_{j=1}^{n-m} (S_{j}(n-j) - S_{j}(m)) \biggr)^{2} \biggr] = 0 \text{ and } \\  \label{continuity_m2}
  \lim_{m \to \infty} \limsup_{n \to \infty} \frac{1}{n} \mathbf{E} \biggl[ \biggl( \sum_{j=n-m+1}^{n} (S_{j}(m) - S_{j}(n-j)) \biggr)^{2} \biggr] = 0.
\end{align}
By stationarity (cf.\ Lemma \ref{lemma:mixing}) the expectation in \eqref{continuity_m2} reduces to
\begin{equation*}
  \mathbf{E} \biggl[ \biggl( \sum_{j=1}^{m} (S_{j}(m) - S_{j}(m-j)) \biggr)^{2} \biggr].
\end{equation*}
Since this is finite and does not depend on $n$, we obtain the desired convergence. Next, recall that $\bar{U}_{j}(k) \coloneqq U_{j}(k) - \bE{U_{j}(k)}$ for all $j,k \in \mathbb{N}_{0}$. The expectation in \eqref{continuity_m1} equals
\begin{equation*}
  4 \sum_{i=1}^{n-m} \sum_{j=1}^{n-m} \sum_{l=m+1}^{n-i} \sum_{k=m+1}^{n-j} \bE{\bar{U}_{i}(l)\bar{U}_{j}(k)}.
\end{equation*}
Using stationarity, we see that this is equal to
\begin{equation*}
  4 \sum_{j=1}^{n-m} \sum_{l=m+1}^{n-j} \sum_{k=m+1}^{n-j} \bE{\bar{U}_{0}(l)\bar{U}_{0}(k)} + 8 \sum_{i=1}^{n-m} \sum_{j=1}^{n-m-i} \sum_{l=m+1}^{n-i} \sum_{k=m+1}^{n-j} \bE{\bar{U}_{0}(l)\bar{U}_{j}(k)}.
\end{equation*}
Dividing by $n$ and taking the limit for $n \to \infty$ we obtain
\begin{equation*}
  4 \sum_{l=m+1}^{\infty} \sum_{k=m+1}^{\infty} \bE{\bar{U}_{0}(l)\bar{U}_{0}(k)} + 8 \sum_{j=1}^{\infty} \sum_{l=m+1}^{\infty} \sum_{k=m+1}^{\infty} \bE{\bar{U}_{0}(l)\bar{U}_{j}(k)}.
\end{equation*}
The limit exists since the series
\begin{equation*}
\sum_{l=1}^{\infty} \sum_{k=1}^{\infty}\bE{\bar{U}_{0}(l) \bar{U}_{0}(k)} \text{ and }
\sum_{j=1}^{\infty} \sum_{l=1}^{\infty} \sum_{k=1}^{\infty} \bE{ \bar{U}_{0}(l) \bar{U}_{j}(k) }
\end{equation*}
are absolutely convergent (this is shown in Appendix \ref{sec:finiteness_of_sigma2}). Thus, by taking the limit for $m \to \infty$, we obtain that the above terms converge to zero yielding \eqref{continuity_m1}. This concludes the proof of (i). Turning to (ii), the law of large numbers for renewal processes (see e.g.\ Corollary 11 of \citet{Serfozo-2009}) yields that $N(T)/T \xrightarrow[T \to \infty]{a.s.} \beta$. As before, it is enough to show that $\mathcal{S}_{N(T)} \xrightarrow[n \to \infty]{d} \mathcal{N}(0, \sigma_{J}^{2})$. To this end, we use a version of Proposition 6.3.9 of \citet{Brockwell-1991} and show that
\begin{align} \label{CLT_for_Sm_bis}
  &\mathcal{\tilde{S}}_{N(T),m} \xrightarrow[T \to \infty]{d} \mathcal{N}(0, \sigma_{J,m}^{2}) \text{ and } \\ \label{continuity_m_bis}
  &\lim_{m \to \infty} \limsup_{T \to \infty} \bP{ \abs{\mathcal{S}_{N(T)} - \mathcal{\tilde{S}}_{N(T),m}} \geq \epsilon } = 0 \text{ for all } \epsilon>0.
\end{align}
Recall that we have already verified \eqref{limiting_variance} yielding that $\mathcal{N}(0, \sigma_{J,m}^{2}) \xrightarrow[m \to \infty]{d} \mathcal{N}(0, \sigma_{J}^{2})$. For all $0<\delta<\beta$ the probability in \eqref{continuity_m_bis} is bounded above by
\begin{equation*}
  \bP{ \abs{\mathcal{S}_{N(T)} - \mathcal{\tilde{S}}_{N(T),m}} \geq \epsilon, N(T) \geq (\beta-\delta)T } + \bP{N(T) < (\beta-\delta)T},
\end{equation*}
where in turn the first term is bounded by
\begin{equation*}
  \sup_{n \geq (\beta-\delta)T} \bP{ \abs{\mathcal{S}_{n} - \mathcal{\tilde{S}}_{n,m}} \geq \epsilon }.
\end{equation*}
Using \eqref{continuity_m} and $N(T)/T \xrightarrow[T \to \infty]{p} \beta$ we obtain \eqref{continuity_m_bis}.

We now turn to \eqref{CLT_for_Sm_bis}. We follow the argument used by \citet{Renyi-1957} to prove Anscombe's theorem \citep{Anscombe-1952, Gut-2012} for sums of i.i.d.\ random variables. In this case, we use a moment inequality for strongly mixing sequences due to \citet{Yokoyama-1980}. Let $0<\epsilon \leq 1/2$ and $A_{T,\epsilon} \coloneqq \{ \abs{N(T)-\beta T} \leq \beta \epsilon T \}$. Since $N(T)/T \xrightarrow[T \to \infty]{p} \beta$, there exists $T_{1}>0$ such that $\bP{A_{T,\epsilon}^{\complement}} \leq \epsilon \text{ for all } T \geq T_{1}$. It follows that for all $x \in \R$ and $T \geq T_{1}$
\begin{equation}
\begin{aligned} \label{clt_ii_1}
  \bP{\{ \mathcal{\tilde{S}}_{N(T),m} \leq x \} \cap A_{T,\epsilon}} &\leq \bP{\mathcal{\tilde{S}}_{N(T),m} \leq x} \\
    &\leq \bP{\{ \mathcal{\tilde{S}}_{N(T),m} \leq x \} \cap A_{T,\epsilon}} + \epsilon.
\end{aligned}  
\end{equation}
Let $n_{1} \coloneqq \lceil \beta T (1-\epsilon) \rceil$ and $n_{2} \coloneqq \lfloor \beta T (1+\epsilon) \rfloor$. Then, we have
\begin{equation} \label{clt_ii_2}
  \bP{\{ \mathcal{\tilde{S}}_{N(T),m} \leq x \} \cap A_{T,\epsilon}} = \sum_{n=n_{1}}^{n_{2}} \bP{\mathcal{\tilde{S}}_{n,m} \leq x, N(T)=n },
\end{equation}
where
\begin{equation*}
  \bP{\mathcal{\tilde{S}}_{n,m} \leq x, N(T)=n } = \bP{\sum_{j=1}^{n_{1}} S_{j}(m) \leq \sqrt{n} x - \sum_{j=n_{1}+1}^{n} S_{j}(m), N(T)=n }.
\end{equation*}
By letting $M \coloneqq \max_{j=n_{1},\dots,n_{2}} \abs{\sum_{j=n_{1}+1}^{n} S_{j}(m)}$ we obtain
\begin{equation}
\begin{aligned} \label{clt_ii_3}
  \bP{\sum_{j=1}^{n_{1}} S_{j}(m) \leq \sqrt{n_{1}} x - M, N(T)=n} &\leq \bP{\mathcal{\tilde{S}}_{m,m} \leq x, N(T)=n } \\
  &\leq \bP{\sum_{j=1}^{n_{1}} S_{j}(m) \leq \sqrt{n_{2}} x + M, N(T)=n }.
\end{aligned}  
\end{equation}
Next, using \ref{H1} and Lemma \ref{lemma:moments_of_U} we see that $\bE{\abs{S_{j}(m)}^{2+\delta/2}} < \infty$ for all $j \in \N$. Since by Lemma \ref{lemma:mixing} $(S_{j}(m))_{j \in \N_{0}}$ is strongly mixing with exponentially decaying
coefficients, condition (3.1) of Theorem 1 of \citet{Yokoyama-1980} is trivially satisfied for all $2<r<2+\delta/2$ (assume w.l.o.g.\ that $r \leq 3$). Let $B_{M,\epsilon} \coloneqq \{ M \leq \epsilon^{1/3} n_{1}^{1/2} \}$. Then, Corollary 1 of \citet{Yokoyama-1980} gives a constant $C_{r}$ such that $\bE{M^{r}} \leq C_{r} (n_{2}-n_{1})^{r/2}$ yielding that
\begin{equation*}
  \bP{B_{M,\epsilon}^{\complement}} \leq C_{r} \epsilon^{-r/3} \biggl( \frac{n_{2}-n_{1}}{n_{1}} \biggr)^{r/2} \leq  2^{r} C_{r} \epsilon^{r/6} \leq 8 C_{r} \epsilon^{1/3}.
\end{equation*}
Using the above inequality in \eqref{clt_ii_1}, \eqref{clt_ii_2}, and \eqref{clt_ii_3} we obtain that
\begin{align*}
  &\bP{\mathcal{\tilde{S}}_{N(T),m} \leq x} \geq \bP{ \{ \mathcal{\tilde{S}}_{n_{1},m} \leq x-\epsilon^{1/3} \} \cap A_{T,\epsilon} \cap B_{M,\epsilon}} \; \text{ and} \\
  &\bP{\mathcal{\tilde{S}}_{N(T),m} \leq x} \leq \bP{ \{ \mathcal{\tilde{S}}_{n_{2},m} \leq (n_{2}/n_{1})^{1/2} x+\epsilon^{1/3} \} \cap A_{T,\epsilon} \cap B_{M,\epsilon}} +8C_{r}\epsilon^{1/3}+\epsilon.
\end{align*}
Since $\bP{(A_{T,\epsilon} \cap B_{M,\epsilon})^{\complement}} \leq \epsilon + 8 C_{r} \epsilon^{1/3}$ we deduce that
\begin{align*}
  \bP{ \mathcal{\tilde{S}}_{n_{1},m} \leq x-\epsilon^{1/3} } - 8 C_{r} \epsilon^{1/3} - \epsilon &\leq \bP{\mathcal{\tilde{S}}_{N(T),m} \leq x} \\
  &\leq \bP{ \mathcal{\tilde{S}}_{n_{2},m} \leq (n_{2}/n_{1})^{1/2} x+\epsilon^{1/3} \}} +8C_{r}\epsilon^{1/3}+\epsilon.
\end{align*}
The result now follows from \eqref{CLT_for_Sm} and the fact that $\lim_{T \to \infty} \frac{n_{2}}{n_{1}} = \frac{1+\epsilon}{1-\epsilon}$ and $\epsilon>0$ is arbitrarily small.
\end{proof}

\section{Asymptotic normality of the estimator} \label{sec:asymptotic_normality_estimator}

In this section we establish consistency and asymptotic normality of the estimators $(\hat{\boldsymbol{\theta}}_{n})_{n \in \mathbb{N}}$ and $(\hat{\boldsymbol{\theta}}_{N(T)})_{T \geq 0}$. In particular, we prove Theorem \ref{theorem:asymptotic_normality} and Corollary \ref{corollary:asymptotic_normality}. We begin by noticing that the estimators are consistent. Using Lemma \ref{lemma:clt} (i) with
\begin{equation*}
  G(u) = \frac{\log(g(u,\boldsymbol{\theta}))}{1+u^{2}},
\end{equation*}
we see that
\begin{equation} \label{consistency_of_K}
    \hat{K}_{n}(\boldsymbol{\theta}) \xrightarrow[n \to \infty]{p} K(\boldsymbol{\theta}).
\end{equation}    
Furthermore, using \eqref{periodogram} and \eqref{truncated_spectral_density} one can show that
\begin{equation*}
  \bE{I_{Z,n}(u)} = \frac{1}{n} \sum_{j=1}^{n} \bar{\phi}_{Z,n-j}(u, \boldsymbol{\theta}_{0}).
\end{equation*}
Since by Lemma \ref{lemma:autocovariance}
\begin{equation*}
  \int_{-\infty}^{\infty} \abs{\gamma_{Y}(h)} \biggl( r(\abs{h}) - \sum_{k=1}^{n-j} f^{\ast k}(\abs{h}) \biggr) \, dh \leq \frac{2C (\mathcal{L}(D))^{n-j+1}}{1-\mathcal{L}(D)},
\end{equation*}
we see that
\begin{equation} \label{expectation_periodogram}
  \bE{I_{Z,n}(u)} = \phi_{Z}(u, \boldsymbol{\theta}_{0}) + O(1/n).
\end{equation}  
Using \ref{H1}-\ref{H4}, \eqref{consistency_of_K}, \eqref{expectation_periodogram}, and \eqref{equicontinuity_log_g} in Appendix \ref{sec:regularity}, the same steps as in the proof of Theorem 2.2 of \citet{Lii-1992} (see also \citet{Ibragimov-1967}) yield that $\hat{\boldsymbol{\theta}}_{n} \xrightarrow[n \to \infty]{p} \boldsymbol{\theta}_{0}$. The proof that $\hat{\boldsymbol{\theta}}_{N(T)}  \xrightarrow[T \to \infty]{p} \boldsymbol{\theta}_{0}$ is similar using that by Lemma \ref{lemma:clt} (ii) with
\begin{equation*}
  G(u) = \frac{\log(g(u,\boldsymbol{\theta}))}{1+u^{2}} \text{ and } G(u) = \frac{1}{1+u^{2}},
\end{equation*}
it holds that
\begin{equation*}
    \hat{K}_{N(T)}(\boldsymbol{\theta}) \xrightarrow[T \to \infty]{p} K(\boldsymbol{\theta}) \text{ and }  \mathbf{E} \biggl[ \int_{-\infty}^{\infty} \frac{I_{Z,N(T)}(u)}{1+u^{2}} \, du \biggr] \xrightarrow[T \to \infty]{} \int_{-\infty}^{\infty} \frac{\phi_{Z}(u, \boldsymbol{\theta}_{0})}{1+u^{2}} \, du.
\end{equation*}
We now turn to the asymptotic normality. We show in Appendix \ref{sec:regularity} that the function $\boldsymbol{\theta} \mapsto \log(g(u,\boldsymbol{\theta})$ is infinitely often differentiable and its partial derivatives are uniformly bounded in absolute value. This allows differentiation under the integral sign in the following (see Appendix \ref{sec:regularity} for more details). In particular, the first order partial derivatives of $K(\cdot)$ and $\hat{K}_{n}(\cdot)$ are given by
\begin{align*}
  K^{(i)}(\boldsymbol{\theta}) &\coloneqq \frac{\partial K(\boldsymbol{\theta})}{\partial \theta_{i}} = \int_{-\infty}^{\infty} \frac{\partial \log(g(u,\boldsymbol{\theta}))}{\partial \theta_{i}} \frac{\phi_{Z}(u, \boldsymbol{\theta}_{0})}{1+u^{2}} \, du \, \text{ and} \\
  \hat{K}_{n}^{(i)}(\boldsymbol{\theta}) &\coloneqq \frac{\partial \hat{K}_{n}(\boldsymbol{\theta})}{\partial \theta_{i}} = \int_{-\infty}^{\infty} \frac{\partial \log(g(u,\boldsymbol{\theta}))}{\partial \theta_{i}} \frac{I_{Z,n}(u)}{1+u^{2}} \, du.
\end{align*}
Similarly the second order partial derivatives are
\begin{align*}
  K^{(i,j)}(\boldsymbol{\theta}) &\coloneqq \frac{\partial^{2} K(\boldsymbol{\theta})}{\partial \theta_{i} \partial \theta_{j}} = \int_{-\infty}^{\infty} \frac{\partial^{2} \log(g(u,\boldsymbol{\theta}))}{\partial \theta_{i} \partial \theta_{j}} \frac{\phi_{Z}(u, \boldsymbol{\theta}_{0})}{1+u^{2}} \, du \text{ and } \\
  \hat{K}_{n}^{(i,j)}(\boldsymbol{\theta}) &\coloneqq \frac{\partial^{2} \hat{K}_{n}(\boldsymbol{\theta})}{\partial \theta_{i} \partial \theta_{j}} = \int_{-\infty}^{\infty} \frac{\partial^{2} \log(g(u,\boldsymbol{\theta}))}{\partial \theta_{i} \partial \theta_{j}} \frac{I_{Z,n}(u)}{1+u^{2}} \, du.
\end{align*}
The matrix $W$ in Theorem \ref{theorem:asymptotic_normality} has elements $w_{ij}=K^{(i,j)}(\boldsymbol{\theta}_{0})$ for $i,j=1,\dots,p+q$. As in (3.23) of \citet{Lii-1992} one can show that
\begin{equation} \label{entries_W}
  w_{ij} = - \int_{-\infty}^{\infty} \frac{\partial \log(g(u,\boldsymbol{\theta}_{0}))}{\partial \theta_{i}} \frac{\partial \log(g(u,\boldsymbol{\theta}_{0}))}{\partial \theta_{j}} \frac{\phi_{Z}(u, \boldsymbol{\theta}_{0})}{1+u^{2}} \, du.
\end{equation}
The entries of the matrix $Q$ are given for $i,j=1,\dots,p+q$ by
\begin{equation} \label{entries_Q}
  q_{ij} = Q_{ij}( \varphi^{(i)}(\cdot,\boldsymbol{\theta}_{0}), \varphi^{(j)}(\cdot,\boldsymbol{\theta}_{0}) ),
\end{equation}
where $Q_{i,j}(\cdot \, , \, \cdot)$ is given by \eqref{entries_Q_G} and 
\begin{equation*}
  \varphi^{(i)}(u,\boldsymbol{\theta}) \coloneqq \frac{\partial}{\partial \theta_{i}} \frac{\log(g(u,\boldsymbol{\theta}))}{(1+u^{2})}.
\end{equation*}
We now prove (i) of Theorem \ref{theorem:asymptotic_normality}. We may proceed along the lines of Theorem 3.1 of \citet{Lii-1992}. By the mean value theorem there exist $\tilde{\boldsymbol{\theta}}_{n}^{(i)}$ along the line joining $\hat{\boldsymbol{\theta}}_{n}$ and $\boldsymbol{\theta}_{0}$ such that
\begin{equation*}
  \hat{K}_{n}^{(i)}(\hat{\boldsymbol{\theta}}_{n}) = \hat{K}_{n}^{(i)}(\boldsymbol{\theta}_{0}) + \sum_{j=1}^{p+q} \hat{K}_{n}^{(i,j)}(\tilde{\boldsymbol{\theta}}_{n}^{(i)}) (\hat{\theta}_{n,j} - \theta_{0,j})
\end{equation*}  
for all $i=1,\dots,p+q$. Since $\hat{\boldsymbol{\theta}}_{n}$ is a point of maximum for $\hat{K}_{n}$, we have that $\hat{K}_{n}^{(i)}(\hat{\boldsymbol{\theta}}_{n})=0$. Multiplying both sides by $\sqrt{n}$ we obtain that
\begin{equation} \label{partial_derivatives_mean_value_theorem}
   \sum_{j=1}^{p+q} \hat{K}_{n}^{(i,j)}(\tilde{\boldsymbol{\theta}}_{n}^{(i)}) \sqrt{n} (\hat{\theta}_{n,j} - \theta_{0,j}) = -\sqrt{n} \hat{K}_{n}^{(i)}(\boldsymbol{\theta}_{0})
\end{equation}
for all $i=1,\dots,p+q$. We show below that
\begin{equation} \label{convergence_second_order_derivatives}
  \hat{K}_{n}^{(i,j)}(\tilde{\boldsymbol{\theta}}_{n}^{(i)}) \xrightarrow[n \to \infty]{p} w_{i,j}
\end{equation}  
and 
\begin{equation} \label{CLT_first_derivative}
  \sqrt{n} (\hat{K}_{n}^{(1)}(\boldsymbol{\theta}_{0}), \dots, \hat{K}_{n}^{(p+q)}(\boldsymbol{\theta}_{0}))^{\top} \xrightarrow[n \to \infty]{d} \mathcal{N}(\boldsymbol{0}, Q).
\end{equation}
Using \eqref{partial_derivatives_mean_value_theorem}, \eqref{convergence_second_order_derivatives}, and \eqref{CLT_first_derivative}, we obtain that
\begin{equation*}
  W \sqrt{n} (\hat{\boldsymbol{\theta}}_{n} - \boldsymbol{\theta}_{0}) \xrightarrow[n \to \infty]{d} \mathcal{N}(\boldsymbol{0}, Q).
\end{equation*}
Since $W$ is symmetric, we conclude that
\begin{equation*}
  \sqrt{n} (\hat{\boldsymbol{\theta}}_{n} - \boldsymbol{\theta}_{0}) \xrightarrow[n \to \infty]{d} \mathcal{N}(\boldsymbol{0}, \Sigma_{0}),
\end{equation*}
where $\Sigma_{0} = W^{-1}QW^{-1}$. To establish \eqref{convergence_second_order_derivatives} we show that
\begin{align} 
  &\abs{\hat{K}_{n}^{(i,j)}(\tilde{\boldsymbol{\theta}}_{n}^{(i)}) - \hat{K}_{n}^{(i,j)}(\boldsymbol{\theta}_{0})} \xrightarrow[n \to \infty]{p} 0 \text{ and } \label{convergence_second_order_derivatives_1} \\
  &\hat{K}_{n}^{(i,j)}(\boldsymbol{\theta}_{0}) \xrightarrow[n \to \infty]{p} K^{(i,j)}(\boldsymbol{\theta}_{0}). \label{convergence_second_order_derivatives_2}
\end{align}
\eqref{convergence_second_order_derivatives_1} is obtained as in (3.20a) of \citet{Lii-1992} using the consistency of the estimator $\hat{\boldsymbol{\theta}}_{n}$, the uniformly boundedness of the third order partial derivatives of the function $\boldsymbol{\theta} \mapsto \log(g(u,\boldsymbol{\theta})$, and that by Lemma \ref{lemma:clt} (i) with $G(u) = 1/(1+u^{2})$
\begin{equation} \label{convergence_s2}
    \hat{s}_{n}^{2}(\boldsymbol{\theta}_{0}) \xrightarrow[n \to \infty]{p} s^{2}(\boldsymbol{\theta}_{0}).
\end{equation}
Similarly, \eqref{convergence_second_order_derivatives_2} is obtained using Lemma \ref{lemma:clt} (i) with
\begin{equation*}
  G(u) = \biggl( \frac{\partial^{2} \log(g(u,\boldsymbol{\theta}_{0}))}{\partial \theta_{i} \partial \theta_{j}} \biggr) \cdot \frac{1}{1+u^{2}}.
\end{equation*}  
We now turn to the proof of \eqref{CLT_first_derivative}. We show that
\begin{equation*}
  \sqrt{n} \begin{pmatrix}
    \hat{K}_{n}^{(1)}(\boldsymbol{\theta}_{0}) - K^{(1)}(\boldsymbol{\theta}_{0}) \\
    \vdots \\
    \hat{K}_{n}^{(p+q)}(\boldsymbol{\theta}_{0}) - K^{(p+q)}(\boldsymbol{\theta}_{0})
\end{pmatrix} \xrightarrow[n \to \infty]{d} \mathcal{N}(\boldsymbol{0}, Q).
\end{equation*}
Let $\alpha_{s} \in \R$ for $s=1,\dots,p+q$. By the Cram\'{e}r–Wold device it suffices to show that $J_{n}$ with
\begin{equation*}
  G(u) = \sum_{s=1}^{p+q} \alpha_{s} \varphi^{(s)}(u,\boldsymbol{\theta}_{0})
\end{equation*}
is asymptotically normal with mean zero and variance $\sum_{s=1}^{p+q} \sum_{t=1}^{p+q} \alpha_{s} q_{s,t} \alpha_{t}$. This follows from Lemma \ref{lemma:clt} (i) and the variance calculation in Appendices \ref{sec:covariance_matrix_Q} (in particular, see \eqref{limiting_variance_sum_Gs} with $G_{s}=\varphi^{(s)}(\cdot,\boldsymbol{\theta}_{0})$). To conclude we notice that as in \citet{Lii-1992}
\begin{align*}
  K^{(i)}(\boldsymbol{\theta}_{0}) &= \int_{-\infty}^{\infty} \varphi^{(i)}(u,\boldsymbol{\theta}_{0}) \, du \\
  &=s^{2}(\boldsymbol{\theta}_{0}) \int_{-\infty}^{\infty} \frac{\partial g(u,\boldsymbol{\theta}_{0})}{\partial \theta_{i}} \frac{1}{1+u^{2}} \, du = 0
\end{align*}
for all $i=1,\dots,p+q$ since
\begin{equation*}
  \int_{-\infty}^{\infty} \frac{g(u,\boldsymbol{\theta}_{0})}{1+u^{2}} \, du = 1.
\end{equation*}
Finally, the proof of Theorem \ref{theorem:asymptotic_normality} (ii) is similar using part (ii) of Lemma \ref{lemma:clt}.
  
\section{Concluding remarks} \label{sec:concluding_remarks}

Under a finite $4+\delta$ moment assumption, we establish consistency and asymptotic normality of a Whittle-type estimator of the parameters of a L\'evy-driven CARMA process sampled at a renewal sequence. Our results rely on Proposition 4.1 of \citet{Brandes-2023}, which ensures that the sampled process is strongly mixing with exponentially decaying coefficients. Since their results require only stationarity and strongly mixing properties we expect that our results carry over to more general classes of processes, including L\'evy-driven moving average processes \citep{Brandes-2019}. However, when the sampling window grows (see Lemma \ref{lemma:clt} (ii) above), our proof relies on Theorem 1 of \citet{Yokoyama-1980} (see also Theorems 2.1-2.2 of \citet{Yang-2007}), which yields a trade-off between moment assumptions and the rate of decay of the mixing coefficients. This means that a stronger moment assumption is required if the mixing coefficients decay only polynomially fast. The investigation of these issues is the subject of ongoing research.

\newpage

\appendix

\section{Regularity of the spectral density} \label{sec:regularity}

In this section we establish regularity for the rescaled spectral density \eqref{function_g}, in particular, equicontinuity and differentiability. We begin by establishing a uniform upper bound on the function $t \mapsto \mathbf{b}^{\top} e^{A t} \Sigma \mathbf{b}$.

\subsection{Uniform upper bounds} \label{sec:uniform_upper_bounds}

We require a uniform upper bound on the matrix exponential of $A$ over $\boldsymbol{\theta} \in \Theta$. To this end, let $\lambda_{1}, \dots, \lambda_{p}$ be the eigenvalues of $A$ and $\alpha(A) = \max_{j=1,\dots,p} Re(\lambda_{j})$ be the spectral abscissa of $A$, where $Re(z)$ is the real part of the complex number $z$. The eigenvalues $\lambda_{j}$ are not necessarily distinct and correspond to the zeros of the polynomial $a(\cdot)$. Since the zeros of a polynomial vary continuously as a function of the coefficients (see e.g.\ \citet{Harris-1987}), $\alpha(A)$ varies continuously as a function of $\boldsymbol{\theta}$. Using \ref{H3}, we see that $\alpha(A)$ is bounded above by a negative constant uniformly over $\boldsymbol{\theta} \in \Theta$. Furthermore, the matrix $A$ is infinitely often differentiable with respect to the parameter $\boldsymbol{\theta}$. This shows that the conditions of Lemma 9.9 of \citet{Khalil-2002} are satisfied. Its proof yields constants $\tilde{C}, D > 0$ such that
\begin{equation} \label{universal_exponential_decrease_exponential_matrix}
  \norm{e^{A t}}, \norm{e^{A^{\top} t}} \leq \tilde{C} e^{-Dt}
\end{equation}
for all $t \geq 0$ and $\boldsymbol{\theta} \in \Theta$. Using \eqref{autocovariance}, \eqref{universal_exponential_decrease_exponential_matrix}, and
\begin{equation} \label{norm_sigma_integrand}
  \norm{e^{A s} \mathbf{e}_{p} \mathbf{e}_{p}^{\top} e^{A^{\top} s}} \leq \norm{e^{A s} \mathbf{e}_{p}}^{2} \leq \norm{e^{A s}}^{2},
\end{equation}
we have    
\begin{equation} \label{norm_sigma}
  \norm{\Sigma} \leq \frac{\tilde{C}^{2}}{2D}.
\end{equation}
Using \eqref{universal_exponential_decrease_exponential_matrix}, \eqref{norm_sigma}, and the fact that $\norm{\mathbf{b}}$ is uniformly bounded over $\Theta$, we obtain a constant $\doubletilde{C}>0$ such that
\begin{equation} \label{exponential_decrease_autocovariance}
  \abs{\mathbf{b}^{\top} e^{A t} \Sigma \mathbf{b}} \leq \doubletilde{C} e^{-Dt}
\end{equation}
for all $t \geq 0$ and $\boldsymbol{\theta} \in \Theta$.

\subsection{Equicontinuity} \label{sec:equicontinuity}

In this section we show the following equicontinuity condition, which is used in the proof of Theorem \ref{theorem:asymptotic_normality}: for all $\epsilon>0$ there exists $\delta>0$ such that
\begin{equation} \label{equicontinuity_g}
  \abs{g(u,\boldsymbol{\theta}^{(1)}) - g(u,\boldsymbol{\theta}^{(2)})} \leq \epsilon
\end{equation}
for all $u \in \R$ and $\boldsymbol{\theta}^{(1)}, \boldsymbol{\theta}^{(2)} \in \Theta$ such that $\norm{\boldsymbol{\theta}^{(1)}-\boldsymbol{\theta}^{(2)}} \leq \delta$. We recall from \eqref{spectral_density_Z} and \eqref{function_g} that
\begin{equation} \label{function_g_as_ratio}
  g(u,\boldsymbol{\theta}) = \frac{ h(\boldsymbol{\theta}) + g_{1}(u,\boldsymbol{\theta}) }{\pi h(\boldsymbol{\theta}) + \int_{-\infty}^{\infty} \frac{g_{1}(u,\boldsymbol{\theta})}{1+u^{2}} \, du},
\end{equation}
where
\begin{equation} \label{h_and_g1}
  h(\boldsymbol{\theta}) \coloneqq \mathbf{b}^{\top} \Sigma \mathbf{b} \text{ and }
  g_{1}(u,\boldsymbol{\theta}) \coloneqq \int_{-\infty}^{\infty} e^{-\iu hu} \, \mathbf{b}^{\top} e^{A \abs{h}} \Sigma \mathbf{b} \, r(\abs{h}) \, dh.
\end{equation}
We notice that
\begin{equation} \label{g1}
  g_{1}(u,\boldsymbol{\theta}) = 2 \int_{0}^{\infty} \cos(tu) g_{2}(t,\boldsymbol{\theta}) r(t) \, dt, \text{ where } g_{2}(t,\boldsymbol{\theta}) \coloneqq \mathbf{b}^{\top} e^{A t} \Sigma \mathbf{b}.
\end{equation}
Using \eqref{autocovariance} and \eqref{universal_exponential_decrease_exponential_matrix} one can show that the function $(t,\boldsymbol{\theta}) \mapsto g_{2}(t,\boldsymbol{\theta})$ is continuous (see Lemma \ref{lemma:g2_infinitely_differentiable} below). Furthermore, \eqref{exponential_decrease_autocovariance} yields that
\begin{equation} \label{exponential_decrease_g2}
  \abs{g_{2}(t,\boldsymbol{\theta})} \leq \doubletilde{C} e^{-Dt}
\end{equation}
for all $t \geq 0$ and $\boldsymbol{\theta} \in \Theta$. By a change of variable we have
\begin{equation*}
  g_{1}(u,\boldsymbol{\theta}) = - 2 \int_{-\pi/u}^{\infty} \cos(tu) g_{2}(t+\pi/u,\boldsymbol{\theta}) r(t+\pi/u) \, dt.
\end{equation*}
By taking the average of this and \eqref{g1} we see that
\begin{align*}
  \abs{g_{1}(u,\boldsymbol{\theta})} &\leq \int_{-\pi/u}^{0} g_{2}(t+\pi/u,\boldsymbol{\theta}) r(t+\pi/u) \, dt \\
  &+\int_{0}^{\infty} \abs{g_{2}(t,\boldsymbol{\theta}) r(t) - g_{2}(t+\pi/u,\boldsymbol{\theta}) r(t+\pi/u)} \, dt.
\end{align*}
Using that $r(\cdot)$ is bounded and continuous by \ref{H2} and $g_{2}(\cdot)$ is continuous and bounded by \eqref{exponential_decrease_g2}, we see that the above RHS converges to zero uniformly in $\boldsymbol{\theta} \in \Theta$ as $u \to \infty$. In particular, for all $\epsilon>0$ there exists $N>0$ such that
\begin{equation} \label{g1_vanishing}
  \sup_{\boldsymbol{\theta} \in \Theta, \, \abs{u} > N} \abs{g_{1}(u,\boldsymbol{\theta})} \leq \frac{\epsilon}{4}.
\end{equation}
Next, using \eqref{exponential_decrease_g2} we see that
\begin{align*}
  \abs{g_{1}(u^{(1)},\boldsymbol{\theta}^{(1)}) - g_{1}(u^{(2)},\boldsymbol{\theta}^{(2)})} &\leq 2 \int_{0}^{M} \abs{\cos(tu^{(1)}) g_{2}(t,\boldsymbol{\theta}^{(1)}) - \cos(tu^{(2)}) g_{2}(t,\boldsymbol{\theta}^{(2)})} r(t) \, dt \\
  &+4 \doubletilde{C} \int_{M}^{\infty} e^{-Dt} r(t) \, dt.
\end{align*}  
Since $r(\cdot)$ is bounded by \ref{H2}, the second term on the RHS can be made arbitrarily small by choosing large $M$. Using the continuity of cosine and $g_{2}(\cdot)$ one can show that the function $(u,\boldsymbol{\theta}) \mapsto g_{1}(u,\boldsymbol{\theta})$ is continuous. In particular, $g_{1}(\cdot)$ is uniformly continuous over $[-N,N] \times \Theta$. Using compactness, one can show that there exists $\delta>0$ such that
\begin{equation*}
  \abs{g_{1}(u,\boldsymbol{\theta}^{(1)}) - g_{1}(u,\boldsymbol{\theta}^{(2)})} \leq \frac{\epsilon}{2}
\end{equation*}
for all $u \in [-N,N]$ and $\boldsymbol{\theta}^{(1)}, \boldsymbol{\theta}^{(2)} \in \Theta$ such that $\norm{\boldsymbol{\theta}^{(1)}-\boldsymbol{\theta}^{(2)}} \leq \delta$. Combining this with \eqref{g1_vanishing}, we see that
\begin{equation} \label{equicontinuity_g1}
  \abs{g_{1}(u,\boldsymbol{\theta}^{(1)}) - g_{1}(u,\boldsymbol{\theta}^{(2)})} \leq \epsilon
\end{equation}
for all $u \in \R$ and $\boldsymbol{\theta}^{(1)}, \boldsymbol{\theta}^{(2)} \in \Theta$ such that $\norm{\boldsymbol{\theta}^{(1)}-\boldsymbol{\theta}^{(2)}} \leq \delta$. Additionally, the continuity of $g_{1}(\cdot)$ implies that the map
\begin{equation} \label{integral_g1}
  \boldsymbol{\theta} \mapsto \int_{-\infty}^{\infty} \frac{g_{1}(u,\boldsymbol{\theta})}{1+u^{2}} \, du  
\end{equation}
is continuous. It is easy to see that $h(\cdot)$ is continuous (c.f.\ Lemma \ref{lemma:g2_infinitely_differentiable} below), and since $\Sigma$ is positive definite, it is bounded away from $0$. Using this, \eqref{integral_g1} and \eqref{equicontinuity_g1} in \eqref{function_g_as_ratio}, we obtain \eqref{equicontinuity_g}. Additionally, since $g(\cdot)$ is uniformly bounded away from zero and infinity, and the logarithm function is continuous, \eqref{equicontinuity_g} holds with $g(\cdot)$ replaced by $\log(g(\cdot))$, that is, for all $\epsilon>0$ there exists $\delta>0$ such that
\begin{equation} \label{equicontinuity_log_g}
    \abs{\log(g(u,\boldsymbol{\theta}^{(1)})) - \log(g(u,\boldsymbol{\theta}^{(2)}))} \leq \epsilon
\end{equation}
for all $u \in \R$ and $\boldsymbol{\theta}^{(1)}, \boldsymbol{\theta}^{(2)} \in \Theta$ such that $\norm{\boldsymbol{\theta}^{(1)}-\boldsymbol{\theta}^{(2)}} \leq \delta$.

\subsection{Partial derivatives} \label{sec:partial_derivatives}

We begin by showing that the function $\boldsymbol{\theta} \mapsto g(u,\boldsymbol{\theta})$ is infinitely often differentiable. Since the denominator in \eqref{function_g_as_ratio} is positive for all $\boldsymbol{\theta} \in \Theta$, it is enough to show that the functions in \eqref{h_and_g1} and \eqref{integral_g1} are infinitely often differentiable with respect to $\boldsymbol{\theta}$.
\begin{lemma} \label{lemma:g2_infinitely_differentiable}
  Assume \ref{H1}-\ref{H3}. Then, the function $(t,\boldsymbol{\theta}) \mapsto g_{2}(t,\boldsymbol{\theta})$ is infinitely often differentiable and its partial derivatives decrease exponentially fast in $t \geq 0$ uniformly over $\boldsymbol{\theta} \in \Theta$.
\end{lemma}
\begin{proof}
  By the product rule for the derivatives of products it is enough to show that the functions $(t,\boldsymbol{\theta}) \mapsto e^{At}$ and $\boldsymbol{\theta} \mapsto \Sigma$ are infinitely often differentiable, the partial derivatives of $(t,\boldsymbol{\theta}) \mapsto e^{At}$ decrease exponentially fast in $t$ uniformly over $\boldsymbol{\theta} \in \Theta$, and the partial derivatives of $\boldsymbol{\theta} \mapsto \Sigma$ are uniformly bounded. It is easy to see that $(t,\boldsymbol{\theta}) \mapsto e^{At}$ is infinitely often differentiable. Furthermore, using \eqref{universal_exponential_decrease_exponential_matrix}, one can show the required exponential decay (notice that the norm of $A$ and its partial derivatives are uniformly bounded in $\boldsymbol{\theta} \in \Theta$ by \ref{H3} and $t^{k}e^{-Dt}$ decreases exponentially for every fixed $k \in \N_{0}$). Turning to the function $\boldsymbol{\theta} \mapsto \Sigma$ the required conditions can be obtained using \eqref{autocovariance}, the differentiability of $(t,\boldsymbol{\theta}) \mapsto e^{At}$ with exponentially fast decay of the partial derivatives, as well as \eqref{norm_sigma_integrand} and \eqref{norm_sigma}.
\end{proof}
In particular, by taking $t=0$ we see that $h(\cdot)$ is infinitely often differentiable and its partial derivatives are uniformly bounded in $\Theta$. Using Lemma \ref{lemma:g2_infinitely_differentiable}, the mean value theorem and dominated convergence theorem, we obtain that the function $\boldsymbol{\theta} \mapsto g_{1}(u,\boldsymbol{\theta})$ is infinitely often differentiable as well. Furthermore, this function and its partial derivatives are uniformly bounded in both $\boldsymbol{\theta} \in \Theta$ and $u \in \R$. In particular, the function $g_{1}(u,\boldsymbol{\theta})/(1+u^{2})$ and its partial derivatives are uniformly bounded by a constant times $1/(1+u^{2})$. Using again the mean value theorem and the dominated convergence theorem, we obtain that the function in \eqref{integral_g1} is infinitely often differentiable. This shows the desired differentiability for the function $\boldsymbol{\theta} \mapsto g(u,\boldsymbol{\theta})$. Furthermore, we have shown in Section \ref{sec:equicontinuity} that $g(u,\boldsymbol{\theta})$ is uniformly bounded away from zero and infinity. We deduce that the function $\boldsymbol{\theta} \mapsto \log(g(u,\boldsymbol{\theta}))$ is infinitely often differentiable and its derivatives are uniformly bounded in absolute value. This shows that it is possible to differentiate the functions $K(\cdot)$ and $\hat{K}_{n}(\cdot)$ under the integral sign with the aid of the mean value theorem and dominated convergence theorem. In particular, we have
\begin{align*}
  K^{(i)}(\boldsymbol{\theta}) &= \frac{\partial K(\boldsymbol{\theta})}{\partial \theta_{i}} = \int_{-\infty}^{\infty} \frac{\partial \log(g(u,\boldsymbol{\theta}))}{\partial \theta_{i}} \frac{\phi_{Z}(u, \boldsymbol{\theta}_{0})}{1+u^{2}} \, du \text{ and} \\
  K^{(i,j)}(\boldsymbol{\theta}) &= \frac{\partial^{2} K(\boldsymbol{\theta})}{\partial \theta_{i} \partial \theta_{j}} = \int_{-\infty}^{\infty} \frac{\partial^{2} \log(g(u,\boldsymbol{\theta}))}{\partial \theta_{i} \partial \theta_{j}} \frac{\phi_{Z}(u, \boldsymbol{\theta}_{0})}{1+u^{2}} \, du,
\end{align*}
where $\theta_{i}$ is the $i^{\texttt{th}}$ component of $\boldsymbol{\theta}$ and $i,j=1,\dots,p+q$. Similarly, we obtain
\begin{align*}
  \hat{K}_{n}^{(i)}(\boldsymbol{\theta}) &= \frac{\partial \hat{K}_{n}(\boldsymbol{\theta})}{\partial \theta_{i}} = \int_{-\infty}^{\infty} \frac{\partial \log(g(u,\boldsymbol{\theta}))}{\partial \theta_{i}} \frac{I_{Z,n}(u)}{1+u^{2}} \, du \text{ and } \\
  \hat{K}_{n}^{(i,j)}(\boldsymbol{\theta}) &= \frac{\partial^{2} \hat{K}_{n}(\boldsymbol{\theta})}{\partial \theta_{i} \partial \theta_{j}} = \int_{-\infty}^{\infty} \frac{\partial^{2} \log(g(u,\boldsymbol{\theta}))}{\partial \theta_{i} \partial \theta_{j}} \frac{I_{Z,n}(u)}{1+u^{2}} \, du.
\end{align*}
Analogously one can obtain higher order partial derivatives.

\section{Moments and variance calculations} \label{sec:moments_and_variance}

In this section we compute mean and covariance of $(U_{j}(k))_{j \in \N_{0}}$ and the limiting variance of $\sqrt{n} J_{n}$.

\subsection{Explicit formulas} \label{sec:explicit_formulas}

\noindent We begin by computing mean and covariance of $(U_{j}(k))_{j \in \N_{0}}$. By stationarity, it is enough to determine $\bE{U_{0}(k)}$ and $\bE{U_{0}(l)U_{j}(k)}$ for all $j,k,l \in \mathbb{N}_{0}$. Conditioning on $\tau$ and using independence of $Y$ and $\tau$ we obtain that
\begin{equation} \label{expectation_U}
  \bE{U_{0}(k)} = \bE{\hat{G}_{R}(\tau_{k}) \gamma_{Y}(\tau_{k})} = \begin{cases}
\hat{G}_{R}(0) \gamma_{Y}(0) &\text{ if } k=0 \\
\int_{0}^{\infty} \hat{G}_{R}(s) \gamma_{Y}(s) f^{\ast k}(s) \, ds &\text{ if } k \in \mathbb{N}.
\end{cases}
\end{equation}
We now turn to the expectation of the product. We recall from \eqref{CARMA} and Lemma 4.1 of \citet{Brandes-2019} that for all $s,t,u \geq 0$
\begin{align}
  M(s,t,u) &\coloneqq \bE{Y(0)Y(s)Y(s+t)Y(s+t+u)} \nonumber \\
  &=N(s,t,u) +\gamma_{Y}(s)\gamma_{Y}(u)+\gamma_{Y}(s+t)\gamma_{Y}(t+u)+\gamma_{Y}(s+t+u)\gamma_{Y}(t), \label{CARMAmoments}
\end{align}
where we have used that $Y$ has mean zero and $N(s,t,u)$ is the $4^{\text{th}}$ order joint cumulant of $Y(0)$, $Y(s)$, $Y(s+t)$ and $Y(s+t+u)$, which is given by
\begin{equation*}
  N(s,t,u) = (\bE{L^{4}(1)} - 3\sigma_{L}^{4}) \int_{0}^{\infty} \mathbf{b}^{\top} e^{Av} \mathbf{e}_{p} \mathbf{b}^{\top} e^{A(v+s)} \mathbf{e}_{p} \mathbf{b}^{\top} e^{A(v+s+t)} \mathbf{e}_{p} \mathbf{b}^{\top} e^{A(v+s+t+u)} \mathbf{e}_{p}  \, dv.
\end{equation*}  
Now, by conditioning on $\tau$ and using independence of $Y$ and $\tau$, we have
\begin{equation*}
  \bE{U_{0}(l) U_{j}(k)} = \bE{\hat{G}_{R}(\tau_{l}) \hat{G}_{R}(\tau_{j+k}-\tau_{j}) \bE{Y(0) Y(\tau_{l}) Y(\tau_{j}) Y(\tau_{j+k}) \, | \, \tau}},
\end{equation*}
where $\bE{Y(0) Y(\tau_{l}) Y(\tau_{j}) Y(\tau_{j+k}) \, | \, \tau}$ is equal to
\begin{equation*}
  \begin{cases}
M(0,0,0) &\text{ if } \mathcal{C}_{1} \coloneqq [j=k=l=0] \\
M(0,0,\tau_{l}) &\text{ if } \mathcal{C}_{2} \coloneqq [j=k=0<l] \\
M(0,0,\tau_{k}) &\text{ if } \mathcal{C}_{3} \coloneqq [j=l=0<k] \\
M(0,\tau_{j},0) &\text{ if } \mathcal{C}_{4} \coloneqq [k=l=0<j] \\    
M(0,\tau_{l},0) &\text{ if } \mathcal{C}_{5} \coloneqq [j=0<k=l] \\    
M(0,\tau_{k},\tau_{l}-\tau_{k}) &\text{ if } \mathcal{C}_{6} \coloneqq [j=0<k<l] \\
M(0,\tau_{l},\tau_{k}-\tau_{l}) &\text{ if } \mathcal{C}_{7} \coloneqq [j=0<l<k] \\
M(\tau_{l},0,0) &\text{ if } \mathcal{C}_{8} \coloneqq [k=0<j=l] \\
M(\tau_{j},0,\tau_{l}-\tau_{j}) &\text{ if } \mathcal{C}_{9} \coloneqq [k=0<j<l] \\
M(\tau_{l},\tau_{j}-\tau_{l},0) &\text{ if } \mathcal{C}_{10} \coloneqq [k=0<l<j] \\
M(0,\tau_{j},\tau_{j+k}-\tau_{j}) &\text{ if } \mathcal{C}_{11} \coloneqq [j,k>0, \, l=0] \\
M(\tau_{l},\tau_{j}-\tau_{l},\tau_{j+k}-\tau_{j}) &\text{ if } \mathcal{C}_{12} \coloneqq [j,k>0, \; 0<l<j] \\
M(\tau_{l},0,\tau_{l+k}-\tau_{l}) &\text{ if } \mathcal{C}_{13} \coloneqq [j,k>0, \; l=j] \\
M(\tau_{j},\tau_{l}-\tau_{j},\tau_{j+k}-\tau_{l}) &\text{ if } \mathcal{C}_{14} \coloneqq [j,k>0, \; j<l<j+k] \\
M(\tau_{j},\tau_{j+k}-\tau_{j},0) &\text{ if } \mathcal{C}_{15} \coloneqq [j,k>0, \; l=j+k] \\
M(\tau_{j},\tau_{j+k}-\tau_{j},\tau_{l}-\tau_{j+k}) &\text{ if } \mathcal{C}_{16} \coloneqq [j,k>0, \; l>j+k]
\end{cases}
\end{equation*}
with the notation $\mathcal{C}_{1}$ representing the condition $j=k=l=0$ on the indices $j,k,l \in \mathbb{N}_{0}$, and similarly for the other cases. Thus, we obtain that $\bE{U_{0}(l) U_{j}(k)}$ is equal to \newline
\begin{equation} \label{covariance_U}
   \begin{cases}
\hat{G}_{R}^{2}(0) M(0,0,0) &\text{ if } \mathcal{C}_{1} \\
\hat{G}_{R}(0) \int \hat{G}_{R}(s) M(0,0,s) \, dF^{\ast l}(s) &\text{ if } \mathcal{C}_{2} \\
\hat{G}_{R}(0) \int \hat{G}_{R}(s) M(0,0,s) \, dF^{\ast k}(s) &\text{ if } \mathcal{C}_{3} \\
\hat{G}_{R}^{2}(0) \int M(0,s,0) \, dF^{\ast j}(s) &\text{ if } \mathcal{C}_{4} \\    
\int \hat{G}_{R}^{2}(s) M(0,s,0) \, dF^{\ast l}(s) &\text{ if } \mathcal{C}_{5} \\
\int \hat{G}_{R}(s) \hat{G}_{R}(s+t) M(0,s,t) \, dF^{\ast k}(s) dF^{\ast (l-k)}(t) &\text{ if } \mathcal{C}_{6} \\
\int \hat{G}_{R}(s) \hat{G}_{R}(s+t) M(0,s,t) \, dF^{\ast l}(s) dF^{\ast (k-l)}(t) &\text{ if } \mathcal{C}_{7} \\
\hat{G}_{R}(0) \int \hat{G}_{R}(s) M(s,0,0) \, dF^{\ast l}(s) &\text{ if } \mathcal{C}_{8} \\
\hat{G}_{R}(0) \int \hat{G}_{R}(s+t) M(s,0,t) \, dF^{\ast j}(s) dF^{\ast (l-j)}(t) &\text{ if } \mathcal{C}_{9} \\
\hat{G}_{R}(0) \int \hat{G}_{R}(s) M(s,t,0) \, dF^{\ast l}(s) dF^{\ast (j-l)}(t) &\text{ if } \mathcal{C}_{10} \\
\hat{G}_{R}(0) \int \hat{G}_{R}(t) M(0,s,t) \, dF^{\ast j}(s) dF^{\ast k}(t) &\text{ if }  \mathcal{C}_{11} \\
\int \hat{G}_{R}(s) \hat{G}_{R}(u) M(s,t,u) \, dF^{\ast l}(s) dF^{\ast (j-l)}(t) dF^{\ast k}(u) &\text{ if } \mathcal{C}_{12} \\
\int \hat{G}_{R}(s) \hat{G}_{R}(t) M(s,0,t) \, dF^{\ast l}(s) dF^{\ast k}(t) &\text{ if } \mathcal{C}_{13} \\
\int \hat{G}_{R}(s+t) \hat{G}_{R}(t+u) M(s,t,u) \, dF^{\ast j}(s) dF^{\ast (l-j)}(t) dF^{\ast (j+k-l)}(u) &\text{ if } \mathcal{C}_{14} \\
\int \hat{G}_{R}(s+t) \hat{G}_{R}(t) M(s,t,0) \, dF^{\ast j}(s) dF^{\ast k}(t) &\text{ if } \mathcal{C}_{15} \\
\int \hat{G}_{R}(s+t+u) \hat{G}_{R}(t) M(s,t,u) \, dF^{\ast j}(s) dF^{\ast k}(t) dF^{\ast (l-k-j)}(u) &\text{ if } \mathcal{C}_{16}.
\end{cases}
\end{equation}

\subsection{Upper bounds} \label{sec:upper_bounds}

We use the formulas in Appendices \ref{sec:explicit_formulas} to obtain upper bounds on the mean and covariance of $(U_{j}(k))_{j \in \N_{0}}$. First, using \eqref{Laplace_transform}, \eqref{expectation_U}, the fact that $\hat{G}_{R}$ is bounded in absolute value by $\norm{G}_{L^{1}}/(2\pi)$, and Lemma \ref{lemma:autocovariance}, we obtain that for all $k \in \mathbb{N}_{0}$
\begin{equation}
  \abs{\bE{U_{0}(k)}} \leq \biggl( \frac{C \norm{G}_{L^{1}}}{2 \pi} \biggr) (\mathcal{L}(D))^{k}. \label{moments_decrease_exponentially}
\end{equation}
Next, we use \eqref{covariance_U} to obtain absolute upper bounds on the covariance of $(U_{j}(k))_{j \in \N_{0}}$. Let $\bar{M}(s,t,u) \coloneqq M(s,t,u) - \gamma_{Y}(s) \gamma_{Y}(u)$. We notice that
\begin{equation} \label{covariance_U_2}
  \bE{U_{0}(l) U_{j}(k)} = \bar{T}_{j,k,l} + T_{j,k,l},
\end{equation}
where the term $\bar{T}_{j,k,l}$ (resp.\ $T_{j,k,l}$) is given by \eqref{covariance_U} with $M(s,t,u)$ replaced by $\bar{M}(s,t,u)$ (resp.\ $\gamma_{Y}(s) \gamma_{Y}(u)$). We derive below absolute upper bounds on both $\bar{T}_{j,k,l}$ and $T_{j,k,l}$. Using the properties of the operator norm and \eqref{universal_exponential_decrease_exponential_matrix}, we have
\begin{equation*}
    \abs{\mathbf{b}^{\top} e^{At} \mathbf{e}_{p}} \leq \norm{e^{A^{\top}t} \mathbf{b}} \leq \tilde{C} e^{-Dt} \norm{\mathbf{b}}
\end{equation*}
for all $t \geq 0$. It follows that
\begin{equation*}
  \frac{N(s,t,u)}{\bE{L^{4}(1)} - 3\sigma_{L}^{4}} \leq (\tilde{C} \norm{\mathbf{b}})^{4} e^{-D(3s+2t+u)} \int_{0}^{\infty} e^{-4Dv} \, dv = \frac{(\tilde{C} \norm{\mathbf{b}})^{4}}{4D} e^{-D(3s+2t+u)}.
\end{equation*}
Using the above inequality and Lemma \ref{lemma:autocovariance} in \eqref{CARMAmoments}, we obtain
\begin{equation*}
  \abs{\bar{M}(s,t,u)} \leq \frac{(\tilde{C} \norm{\mathbf{b}})^{4} (\bE{L^{4}(1)} - 3\sigma_{L}^{4})}{4D} e^{-D(3s+2t+u)} + 2 C^{2} e^{-D(s+2t+u)}.
\end{equation*}
Thus, there is a positive constant $\bar{C}$ such that
\begin{equation}
  \abs{\bar{M}(s,t,u)} \leq \bar{C} e^{-D(s+2t+u)} \label{upper_bound_Mbar}
\end{equation}
for all $s,t,u \geq 0$. Using \eqref{upper_bound_Mbar} and $\mathcal{L}(2D) \leq \mathcal{L}(D)$, one can show that
\begin{equation} \label{upper_bound_Tbar}
  \abs{\bar{T}_{j,k,l}} \leq \biggl( \frac{\bar{C} \norm{G}_{L^{1}}}{2\pi} \biggr)^{2} \bar{\mathcal{T}}_{j,k,l}(D),
\end{equation}
where
\begin{equation*}
  \bar{\mathcal{T}}_{j,k,l}(D) \coloneqq \begin{cases}
    (\mathcal{L}(D))^{l} & \text{ if } \mathcal{C}_{1} \lor \mathcal{C}_{2} \lor \mathcal{C}_{5} \lor \mathcal{C}_{8} \lor \mathcal{C}_{9} \lor \mathcal{C}_{16} \\
    (\mathcal{L}(D))^{k} & \text{ if } \mathcal{C}_{3} \lor \mathcal{C}_{6} \lor \mathcal{C}_{7} \\
    (\mathcal{L}(D))^{j} & \text{ if } \mathcal{C}_{4} \lor \mathcal{C}_{10} \\
    (\mathcal{L}(D))^{j+k} & \text{ if } \mathcal{C}_{11} \lor \mathcal{C}_{12} \lor \mathcal{C}_{13} \lor \mathcal{C}_{14} \lor \mathcal{C}_{15}.
\end{cases}
\end{equation*}
Next, we recall from \eqref{covariance_U_2} that
\begin{equation} \label{terms_covariance}
  \bE{\bar{U}_{0}(l) \bar{U}_{j}(k)} = \bar{T}_{j,k,l} + T_{j,k,l} - \bE{U_{0}(l)} \bE{U_{j}(k)}.
\end{equation}
Using \eqref{expectation_U} we notice that
\begin{equation*}
  \bE{U_{0}(l)} \bE{U_{j}(k)} = T_{j,k,l}
\end{equation*}
whenever the condition $\mathbf{C} \coloneqq \mathcal{C}_{1} \lor \mathcal{C}_{2} \lor \mathcal{C}_{3} \lor \mathcal{C}_{4} \lor \mathcal{C}_{8} \lor \mathcal{C}_{10} \lor \mathcal{C}_{11} \lor \mathcal{C}_{12} \lor \mathcal{C}_{13}$ holds. Thus, in this case,
\begin{equation*}
  \bE{\bar{U}_{0}(l) \bar{U}_{j}(k)} = \bar{T}_{j,k,l}
\end{equation*}
and the upper bound \eqref{upper_bound_Tbar} holds.

In general, we can obtain an upper bound on $\bE{\bar{U}_{0}(l) \bar{U}_{j}(k)}$ by deriving upper bounds on both $T_{j,k,l}$ and $\bE{U_{0}(l)} \bE{U_{j}(k)}$. Specifically, Lemma \ref{lemma:autocovariance} implies that
\begin{equation} \label{product_covariances}
  \abs{\gamma_{Y}(s) \gamma_{Y}(u)} \leq C^{2} e^{-D(s+u)}
\end{equation}  
for all $s,u \geq 0$. It follows that
\begin{equation*}
  \abs{T_{j,k,l}} \leq \biggl( \frac{C \norm{G}_{L^{1}}}{2\pi} \biggr)^{2} \mathcal{T}_{j,k,l}(D),
\end{equation*}
where 
\begin{equation*}
  \mathcal{T}_{j,k,l}(D) \coloneqq \begin{cases}
(\mathcal{L}(D))^{l}    &\text{ if } [k=0] = \mathcal{C}_{1} \lor \mathcal{C}_{2} \lor \mathcal{C}_{4} \lor \mathcal{C}_{8} \lor \mathcal{C}_{9} \lor \mathcal{C}_{10} \\
(\mathcal{L}(D))^{\abs{k-l}}    &\text{ if } [k>0, \; j=0] = \mathcal{C}_{3} \lor \mathcal{C}_{5} \lor \mathcal{C}_{6} \lor \mathcal{C}_{7} \\
(\mathcal{L}(D))^{k+l}    &\text{ if } [j,k>0, \; 0 \leq l \leq j] = \mathcal{C}_{11} \lor \mathcal{C}_{12} \lor \mathcal{C}_{13} \\
(\mathcal{L}(D))^{2j+k-l}    &\text{ if } [j,k>0, \; j < l \leq j+k] = \mathcal{C}_{14} \lor \mathcal{C}_{15} \\
(\mathcal{L}(D))^{l-k}    &\text{ if } [j,k>0, \; l>j+k] = \mathcal{C}_{16}.
\end{cases}    
\end{equation*}
Additionally, \eqref{moments_decrease_exponentially} yields that
\begin{equation} \label{upper_bound_product_expectations}
  \abs{\bE{U_{0}(l)} \bE{U_{j}(k)}} \leq C_{1} (\mathcal{L}(D))^{k+l},
\end{equation}
where $C_{1} \coloneqq (C \norm{G}_{L^{1}}/(2\pi))^{2}$. Since $\mathcal{T}_{j,k,l}(D)$ is larger than both $\bar{\mathcal{T}}_{j,k,l}(D)$ and $(\mathcal{L}(D))^{k+l}$, we see from \eqref{terms_covariance} and the above upper bounds that
\begin{equation} \label{exponential_decay}
  \abs{\bE{\bar{U}_{0}(l) \bar{U}_{j}(k)}} \leq C_{2} \mathcal{T}_{j,k,l}(D)
\end{equation}
for some constant $C_{2}$.

By combining \eqref{upper_bound_Tbar} when condition $\mathbf{C}$ holds and \eqref{exponential_decay} when condition $\mathbf{C}$ does not hold, we conclude that for some constant $C_{3}$
\begin{equation} \label{exponential_decay_2}
  \abs{\bE{\bar{U}_{0}(l) \bar{U}_{j}(k)}} \leq C_{3} \, \mathcal{U}_{j,k,l}(D),
\end{equation}
where
\begin{equation*}
  \mathcal{U}_{j,k,l}(D) \coloneqq \begin{cases}
(\mathcal{L}(D))^{l}    &\text{ if } \mathcal{C}_{1} \lor \mathcal{C}_{2} \lor \mathcal{C}_{8} \lor \mathcal{C}_{9} \\
(\mathcal{L}(D))^{k}    &\text{ if } \mathcal{C}_{3} \\
(\mathcal{L}(D))^{j}    &\text{ if } \mathcal{C}_{4} \lor \mathcal{C}_{10} \\
(\mathcal{L}(D))^{\abs{k-l}}    &\text{ if } \mathcal{C}_{5} \lor \mathcal{C}_{6} \lor \mathcal{C}_{7} \\
(\mathcal{L}(D))^{j+k}    &\text{ if } \mathcal{C}_{11} \lor \mathcal{C}_{12} \lor \mathcal{C}_{13} \\
(\mathcal{L}(D))^{2j+k-l}    &\text{ if } \mathcal{C}_{14} \lor \mathcal{C}_{15} \\
(\mathcal{L}(D))^{l-k}    &\text{ if } \mathcal{C}_{16}.
\end{cases}    
\end{equation*}
However, we need better estimates when $\mathcal{C}_{5} \lor \mathcal{C}_{6} \lor \mathcal{C}_{7} \lor \mathcal{C}_{14} \lor \mathcal{C}_{15} \lor \mathcal{C}_{16}$ holds. For the condition $\mathcal{C}_{5} \lor \mathcal{C}_{6} \lor \mathcal{C}_{7}$, using \eqref{terms_covariance}, we see that \eqref{upper_bound_Tbar} and \eqref{upper_bound_product_expectations} yield a sufficiently fast rate of decay for our purposes. We now study the term
\begin{equation*}
T_{j,k,l} =  \begin{cases}  
\gamma_{Y}^{2}(0) \int_{0}^{\infty} \hat{G}_{R}^{2}(s) f^{\ast l}(s) \, ds &\text{ if } \mathcal{C}_{5} \\
\gamma_{Y}(0) \int_{0}^{\infty} \int_{0}^{\infty} \hat{G}_{R}(s) \hat{G}_{R}(s+t) \gamma_{Y}(t) f^{\ast k}(s) f^{\ast (l-k)}(t) \, ds \, dt &\text{ if } \mathcal{C}_{6} \\
\gamma_{Y}(0) \int_{0}^{\infty}\int_{0}^{\infty} \hat{G}_{R}(s) \hat{G}_{R}(s+t) \gamma_{Y}(t) f^{\ast l}(s) f^{\ast (k-l)}(t) \, ds \, dt &\text{ if } \mathcal{C}_{7}.
\end{cases}
\end{equation*}
We deduce that
\begin{align}
  &\sum_{l=1}^{\infty} T_{j,k,l} = \gamma_{Y}^{2}(0) \int_{0}^{\infty} \hat{G}_{R}^{2}(s) r(s) \, ds  &\text{ if } \mathcal{C}_{5} \label{C5} \\
  &\sum_{k=1}^{\infty} \sum_{l=k+1}^{\infty} T_{j,k,l} = \gamma_{Y}(0) \int_{0}^{\infty} \int_{0}^{\infty} \hat{G}_{R}(s) \hat{G}_{R}(s+t) \gamma_{Y}(t) r(s) r(t) \, ds \, dt &\text{ if } \mathcal{C}_{6} \label{C6} \\
  &\sum_{l=1}^{\infty} \sum_{k=l+1}^{\infty} T_{j,k,l} = \gamma_{Y}(0) \int_{0}^{\infty} \int_{0}^{\infty} \hat{G}_{R}(s) \hat{G}_{R}(s+t) \gamma_{Y}(t) r(s) r(t) \, ds \, dt &\text{ if } \mathcal{C}_{7}. \label{C7}
\end{align}
We recall that the renewal density $r(\cdot)$ is bounded since $f(\cdot)$ is bounded by \ref{H2} (see Theorem 5.26 of \citet{Alsmeyer-2012}). Since $\hat{G}_{R}(\cdot) \in L^{2}(\mathbb{R})$ and $\gamma_{Y}(\cdot)$ decreases exponentially fast by Lemma \ref{lemma:autocovariance}, we conclude that the above integrals are absolutely convergent. Now suppose that $\mathcal{C}_{14} \lor \mathcal{C}_{15} \lor \mathcal{C}_{16}$ holds. Then, \eqref{upper_bound_product_expectations} yields a sufficiently fast rate of decay. We study the term $\bE{U_{0}(l) U_{j}(k)}$ using \eqref{covariance_U}. It follows from \eqref{upper_bound_Mbar} and \eqref{product_covariances} that for some constant $C_{4}$
\begin{equation}
  \abs{\bar{M}(s,t,u)} \leq C_{4} e^{-D(s+u)} \label{upper_bound_M}
\end{equation}
for all $s,t,u \geq 0$. Using \eqref{upper_bound_M} in \eqref{covariance_U}, we obtain that $\abs{\bE{U_{0}(l) U_{j}(k)}}/C_{4}$ is bounded above by
\begin{equation*}
\begin{cases}
\int_{0}^{\infty} \int_{0}^{\infty} \int_{0}^{\infty} \hat{G}_{R}(s+t) \hat{G}_{R}(t+u) e^{-D(s+u)} f^{\ast j}(s) f^{\ast (l-j)}(t) f^{\ast (j+k-l)}(u) \, ds \, dt \, du &\text{ if } \mathcal{C}_{14} \\
\int_{0}^{\infty} \int_{0}^{\infty} \hat{G}_{R}(s+t) \hat{G}_{R}(t) e^{-Ds} f^{\ast j}(s) f^{\ast k}(t) \, ds \, dt &\text{ if } \mathcal{C}_{15} \\
\int_{0}^{\infty} \int_{0}^{\infty} \int_{0}^{\infty} \hat{G}_{R}(s+t+u) \hat{G}_{R}(t) e^{-D(s+u)} f^{\ast j}(s) f^{\ast k}(t) f^{\ast (l-k-j)}(u) \, ds \, dt \, du &\text{ if } \mathcal{C}_{16}.
\end{cases}
\end{equation*}
Using that $f(\cdot)$, $r(\cdot)$ are bounded and $\hat{G}_{R}(\cdot) \in L^{2}(\mathbb{R})$, one can show that 
\begin{align}
  &\sum_{j=1}^{\infty} \sum_{k=2}^{\infty} \sum_{l=j+1}^{j+k-1} \abs{\bE{U_{0}(l) U_{j}(k)}} < \infty \label{C14} \\
  &\sum_{j=1}^{\infty} \sum_{k=1}^{\infty} \abs{\bE{U_{0}(j+k) U_{j}(k)}} < \infty \text{ and } \label{C15} \\
  &\sum_{j=1}^{\infty} \sum_{k=1}^{\infty} \sum_{l=j+k+1}^{\infty} \abs{\bE{U_{0}(l) U_{j}(k)}} < \infty. \label{C16}
\end{align}

\subsection{Finiteness of $\sigma_{J}^{2}$} \label{sec:finiteness_of_sigma2}

We use the formulas and upper bounds in Appendices \ref{sec:explicit_formulas} and \ref{sec:upper_bounds} to show that the limiting variance $\sigma_{J}^{2}$ in Lemma \ref{lemma:clt} is finite. We recall that $\sigma_{J}^{2} = \sigma_{J,\infty}^{2}$, where $\sigma_{J,\infty}^{2}$ is given by \eqref{sigma_m} with $m=\infty$, that is,
\begin{equation*}
\begin{split}  
  \sigma_{J}^{2}=&\bE{\bar{U}_{0}^{2}(0)} + 4 \sum_{l=0}^{\infty} \sum_{k=1}^{\infty}\bE{\bar{U}_{0}(l) \bar{U}_{0}(k)} + 2 \sum_{j=1}^{\infty} \bE{ \bar{U}_{0}(0) \bar{U}_{j}(0) } \\
  + &4 \sum_{j=1}^{\infty} \sum_{k=1}^{\infty} \bE{\bar{U}_{0}(0) \bar{U}_{j}(k) } + 4 \sum_{j=1}^{\infty} \sum_{l=1}^{\infty} \bE{ \bar{U}_{0}(l) \bar{U}_{j}(0) } + 8 \sum_{j=1}^{\infty} \sum_{k=1}^{\infty} \sum_{l=1}^{\infty} \bE{ \bar{U}_{0}(l) \bar{U}_{j}(k) }.
\end{split}  
\end{equation*}
We show that the above series are absolutely convergent. We begin with the second term and consider the case $l=0$. \eqref{exponential_decay_2} implies that
\begin{equation*}
  \sum_{k=1}^{\infty} \abs{\bE{\bar{U}_{0}(0) \bar{U}_{0}(k)}} < \infty.
\end{equation*}
Next, we have from \eqref{terms_covariance} that
\begin{equation*}
  \abs{\bE{\bar{U}_{0}(l) \bar{U}_{j}(k)}} \leq \abs{\bar{T}_{j,k,l}} + \abs{T_{j,k,l}} + \abs{\bE{U_{0}(l)} \bE{U_{j}(k)}}.
\end{equation*}
Using this with $j=0$, \eqref{upper_bound_Tbar}, \eqref{upper_bound_product_expectations}, \eqref{C5} and \eqref{C7} imply that
\begin{equation*}
  \sum_{l=1}^{\infty} \sum_{k=1}^{\infty} \abs{\bE{\bar{U}_{0}(l) \bar{U}_{0}(k)}} = \sum_{l=1}^{\infty} \abs{\bE{\bar{U}_{0}^{2}(l)}} + 2  \sum_{l=1}^{\infty} \sum_{k=l+1}^{\infty} \abs{\bE{\bar{U}_{0}(l) \bar{U}_{0}(k)}} < \infty.
\end{equation*}
Turning to the third and fourth term, \eqref{exponential_decay_2} yields that
\begin{equation*}
  \sum_{j=1}^{\infty} \abs{\bE{ \bar{U}_{0}(0) \bar{U}_{j}(0) }} < \infty \text{ and} \sum_{j=1}^{\infty} \sum_{k=1}^{\infty} \abs{\bE{\bar{U}_{0}(0) \bar{U}_{j}(k) }} < \infty.
\end{equation*}
We now turn to the fifth term and notice that
\begin{equation*}
  \sum_{j=1}^{\infty} \sum_{l=1}^{\infty} \abs{\bE{ \bar{U}_{0}(l) \bar{U}_{j}(0) }}
\end{equation*}
is equal to
\begin{equation*}
  \sum_{l=1}^{\infty} \abs{\bE{ \bar{U}_{0}(l) \bar{U}_{l}(0) }} + \sum_{j=1}^{\infty} \sum_{l=1}^{j-1} \abs{\bE{ \bar{U}_{0}(l) \bar{U}_{j}(0) }} + \sum_{j=1}^{\infty} \sum_{l=j+1}^{\infty} \abs{\bE{ \bar{U}_{0}(l) \bar{U}_{j}(0) }}.
\end{equation*}
\eqref{exponential_decay_2} implies that each of this summations is finite. Finally, we consider
\begin{equation*}
  \sum_{j=1}^{\infty} \sum_{k=1}^{\infty} \sum_{l=1}^{\infty} \abs{\bE{ \bar{U}_{0}(l) \bar{U}_{j}(k) }}.
\end{equation*}
We split the sum over index $l$ as $\sum_{l=1}^{\infty} = \sum_{l=1}^{j} + \sum_{l=j+1}^{\infty}$. Using again \eqref{exponential_decay_2} we see that
\begin{equation*}
  \sum_{j=1}^{\infty} \sum_{k=1}^{\infty} \sum_{l=1}^{j} \abs{\bE{ \bar{U}_{0}(l) \bar{U}_{j}(k) }} = \sum_{j=1}^{\infty} \sum_{k=1}^{\infty} \sum_{l=1}^{j-1} \abs{\bE{ \bar{U}_{0}(l) \bar{U}_{j}(k) }} + \sum_{j=1}^{\infty} \sum_{k=1}^{\infty} \abs{\bE{ \bar{U}_{0}(j) \bar{U}_{j}(k) }} < \infty.
\end{equation*}
We are left with the term
\begin{equation*}
  \sum_{j=1}^{\infty} \sum_{k=1}^{\infty} \sum_{l=j+1}^{\infty} \abs{\bE{ \bar{U}_{0}(l) \bar{U}_{j}(k) }}.
\end{equation*}
Since $\bE{ \bar{U}_{0}(l) \bar{U}_{j}(k) } = \bE{ U_{0}(l) U_{j}(k) } - \bE{ U_{0}(l)} \bE{ U_{0}(k) }$ and by \eqref{upper_bound_product_expectations}
\begin{equation*}
  \sum_{j=1}^{\infty} \sum_{k=1}^{\infty} \sum_{l=j+1}^{\infty} \abs{\bE{ U_{0}(l)} \bE{ U_{0}(k) }} < \infty,
\end{equation*}
it suffices to show that
\begin{equation*}
  \sum_{j=1}^{\infty} \sum_{k=1}^{\infty} \sum_{l=j+1}^{\infty} \abs{\bE{ U_{0}(l) U_{j}(k) }} < \infty.
\end{equation*}
Now, this follows from \eqref{C14}, \eqref{C15}, and \eqref{C16}.
  
\subsection{Variance of $\sqrt{n} J_{n}$} \label{sec:variance_of_nJn}

In this section we compute the variance of
\begin{equation*}
  \sqrt{n} J_{n} = \frac{1}{\sqrt{n}} \sum_{j=1}^{n} \sum_{k=-(n-j)}^{n-j} U_{j}(\abs{k}),
\end{equation*}
where we recall that $\bar{U}_{j}(k) \coloneqq U_{j}(k) - \bE{U_{j}(k)}$ for all $j,k \in \mathbb{N}_{0}$. Using \eqref{Sm} we have
\begin{align*}
  \bV{nJ_{n}} &= \mathbf{Var} \biggl[ \sum_{j=1}^{n} S_{j}(n-j) \biggr] \\
  &= \sum_{i=1}^{n} \biggl( \bE{S_{i}^{2}(n-i)} + 2 \sum_{j=1}^{n-i} \bE{S_{i}(n-i) S_{i+j}(n-(i+j))} \biggr).
\end{align*}
By stationarity, we see that
\begin{equation*}
  \bE{S_{i}(n-i) S_{i+j}(n-(i+j))} = \bE{S_{0}(n-i) S_{j}(n-(i+j))}
\end{equation*}  
and, using again \eqref{Sm}, we obtain that this is equal to
\begin{equation} \label{covariance_Sj}
\begin{split}  
  &\bE{\bar{U}_{0}(0) \bar{U}_{j}(0)} + 2 \sum_{k=1}^{n-(i+j)} \bE{\bar{U}_{0}(0) \bar{U}_{j}(k)} \\
  + 2 &\sum_{l=1}^{n-i} \bE{\bar{U}_{0}(l) \bar{U}_{j}(0)} + 4 \sum_{l=1}^{n-i} \sum_{k=1}^{n-(i+j)} \bE{\bar{U}_{0}(l) \bar{U}_{j}(k) }.
\end{split}  
\end{equation}
In particular, we have
\begin{equation*}
  \bE{S_{i}^{2}(n-i)} = \bE{S_{0}^{2}(n-i)} = \bE{\bar{U}_{0}^{2}(0)} + 4 \sum_{l=0}^{n-i} \sum_{k=1}^{n-i} \bE{\bar{U}_{0}(l) \bar{U}_{0}(k) }.
\end{equation*}  
We deduce that $\bV{nJ_{n}}$ is equal to
\begin{equation} \label{variance_nJn}
\begin{split}  
  &\sum_{i=1}^{n} \biggl( \bE{\bar{U}_{0}^{2}(0)} + 4 \sum_{l=0}^{n-i} \sum_{k=1}^{n-i} \bE{\bar{U}_{0}(l) \bar{U}_{0}(k)} + 2 \sum_{j=1}^{n-i} \biggl( \bE{ \bar{U}_{0}(0) \bar{U}_{j}(0) } \\
  + 2 &\sum_{k=1}^{n-(i+j)} \bE{\bar{U}_{0}(0) \bar{U}_{j}(k) } + 2 \sum_{l=1}^{n-i} \bE{ \bar{U}_{0}(l) \bar{U}_{j}(0) } + 4 \sum_{l=1}^{n-i} \sum_{k=1}^{n-(i+j)} \bE{ \bar{U}_{0}(l) \bar{U}_{j}(k) } \biggr) \biggr).
\end{split}  
\end{equation}
Thus,
\begin{equation*}
  \lim_{n \to \infty} \bV{\sqrt{n} J_{n}} = \sigma_{J,\infty}^{2}=\sigma_{J}^{2},
\end{equation*}
where $\sigma_{J,\infty}^{2}$ is given by \eqref{sigma_m}. We notice that the above convergence is implied by Lemma \ref{lemma:clt}, but the above computations are explicit.

\subsection{Covariance matrix $Q$} \label{sec:covariance_matrix_Q}

In this section we explicitly compute the covariance matrix $Q$. We first notice that the quantities $J_{n}=J_{n}(G)$, $\sigma_{J}^{2}=\sigma_{J}^{2}(G)$, $S_{j}(m)=S_{j}(m,G)$, $U_{j}(k)=U_{j}(k,G)$, and $\bar{U}_{j}(k)=\bar{U}_{j}(k,G)$ all depend on the function $G \in L^{1}(\R) \cap L^{2}(\R)$. In the previous section we have computed
\begin{equation*}
  \bV{nJ_{n}(G)} = \mathbf{E} \biggl[ \biggl( \sum_{j=1}^{n} S_{j}(n-j,G) \biggr)^{2} \biggr]
\end{equation*}
and shown that
\begin{equation*}
  \lim_{n \to \infty} \bV{\sqrt{n} J_{n}(G)} = \sigma_{J}^{2}(G).
\end{equation*}
For $s=1,\dots,p+q$ let $\alpha_{s} \in \R$ and $G_{s} \in L^{1}(\R) \cap L^{2}(\R)$. We compute below
\begin{equation*}
  \mathbf{Var} \biggl[ nJ_{n}(\sum_{s=1}^{p+q} \alpha_{s} G_{s}) \biggr] = \sum_{s=1}^{p+q} \sum_{t=1}^{p+q} \alpha_{s} \alpha_{t} \mathbf{E} \biggl[ \biggl( \sum_{i=1}^{n} S_{i}(n-i, G_{s}) \biggr) \biggl( \sum_{j=1}^{n} S_{j}(n-j, G_{t}) \biggr) \biggr]
\end{equation*}
and determine $\lim_{n \to \infty} \bV{\sqrt{n} J_{n}(\sum_{s=1}^{p+q} \alpha_{s} G_{s})}$. Using stationarity we see that
\begin{align*}
   n Q_{s,t}(G_{s},G_{t},n) &\coloneqq \mathbf{E} \biggl[ \biggl( \sum_{i=1}^{n} S_{i}(n-i, G_{s}) \biggr) \biggl( \sum_{j=1}^{n} S_{j}(n-j, G_{t}) \biggr) \biggr] \\
  &= \sum_{i=1}^{n} \bE{S_{0}(n-i, G_{s}) S_{0}(n-i, G_{t})} \\
  &+ \sum_{i=1}^{n} \sum_{j=1}^{n-i} \bE{S_{0}(n-i, G_{s}) S_{j}(n-(i+j), G_{t})} \\
  &+ \sum_{j=1}^{n} \sum_{i=1}^{n-j} \bE{S_{i}(n-(i+j), G_{s}) S_{0}(n-j, G_{t})}.
\end{align*}
Next, using \eqref{covariance_Sj} we see that $\bE{S_{0}(n-i, G_{s}) S_{j}(n-(i+j), G_{t})}$ is equal to
\begin{align*}  
  &\bE{\bar{U}_{0}(0,G_{s}) \bar{U}_{j}(0,G_{t})} + 2 \sum_{k=1}^{n-(i+j)} \bE{\bar{U}_{0}(0,G_{s}) \bar{U}_{j}(k,G_{t})} \\
  + 2 &\sum_{l=1}^{n-i} \bE{\bar{U}_{0}(l,G_{s}) \bar{U}_{j}(0,G_{t})} + 4 \sum_{l=1}^{n-i} \sum_{k=1}^{n-(i+j)} \bE{\bar{U}_{0}(l,G_{s}) \bar{U}_{j}(k,G_{t}) }.
\end{align*}
We deduce that
\begin{align*}  
  n Q_{s,t}(G_{s},G_{t},n) = \sum_{i=1}^{n} \biggl( &\bE{\bar{U}_{0}(0,G_{s}) \bar{U}_{0}(0,G_{t})} + 2 \sum_{k=1}^{n-i} \bE{\bar{U}_{0}(0,G_{s}) \bar{U}_{0}(k,G_{t})} \\
  + 2 &\sum_{l=1}^{n-i} \bE{\bar{U}_{0}(l,G_{s}) \bar{U}_{0}(0,G_{t})} + 4 \sum_{l=1}^{n-i} \sum_{k=1}^{n-i} \bE{\bar{U}_{0}(l,G_{s}) \bar{U}_{0}(k,G_{t}) } \\
  + \sum_{j=1}^{n-i} \biggl( &\bE{\bar{U}_{0}(0,G_{s}) \bar{U}_{j}(0,G_{t})} + 2 \sum_{k=1}^{n-(i+j)} \bE{\bar{U}_{0}(0,G_{s}) \bar{U}_{j}(k,G_{t})} \\
  + 2 &\sum_{l=1}^{n-i} \bE{\bar{U}_{0}(l,G_{s}) \bar{U}_{j}(0,G_{t})} + 4 \sum_{l=1}^{n-i} \sum_{k=1}^{n-(i+j)} \bE{\bar{U}_{0}(l,G_{s}) \bar{U}_{j}(k,G_{t})} \biggr) \\
  + \sum_{i=1}^{n-j} \biggl( &\bE{\bar{U}_{i}(0,G_{s}) \bar{U}_{0}(0,G_{t})} + 2 \sum_{k=1}^{n-(i+j)} \bE{\bar{U}_{i}(k,G_{s}) \bar{U}_{0}(0,G_{t})} \\
  + 2 &\sum_{l=1}^{n-j} \bE{\bar{U}_{i}(0,G_{s}) \bar{U}_{0}(l,G_{t})} + 4 \sum_{l=1}^{n-j} \sum_{k=1}^{n-(i+j)} \bE{\bar{U}_{i}(k,G_{s}) \bar{U}_{0}(l,G_{t})} \biggr) \biggr).
\end{align*}
In the special case that $G_{s}=G_{t}=G$ we see that $n Q_{s,t}(G,G,n)$ is equal to \eqref{variance_nJn}. The above calculations show that
\begin{equation*}
  \mathbf{Var} \biggl[ \sqrt{n} J_{n}(\sum_{s=1}^{p+q} \alpha_{s} G_{s}) \biggr] = \sum_{s=1}^{p+q} \sum_{t=1}^{p+q} \alpha_{s} \alpha_{t} Q_{s,t}(G_{s},G_{t},n)
\end{equation*}
We conclude that
\begin{equation} \label{limiting_variance_sum_Gs}
  \lim_{n \to \infty} \mathbf{Var} \biggl[ \sqrt{n} J_{n}(\sum_{s=1}^{p+q} \alpha_{s} G_{s}) \biggr] = \sum_{s=1}^{p+q} \sum_{t=1}^{p+q} \alpha_{s} \alpha_{t} Q_{s,t}(G_{s},G_{t}),
\end{equation}
where
\begin{equation} \label{entries_Q_G}
\begin{split}  
  Q_{s,t}(G_{s},G_{t}) &\coloneqq  \lim_{n \to \infty} Q_{s,t}(G_{s},G_{t},n) \\
  =&\bE{\bar{U}_{0}(0,G_{s}) \bar{U}_{0}(0,G_{t})} + 2 \sum_{k=1}^{\infty} \bE{\bar{U}_{0}(0,G_{s}) \bar{U}_{0}(k,G_{t})} \\
  + 2 &\sum_{l=1}^{\infty} \bE{\bar{U}_{0}(l,G_{s}) \bar{U}_{0}(0,G_{t})} + 4 \sum_{l=1}^{\infty} \sum_{k=1}^{\infty} \bE{\bar{U}_{0}(l,G_{s}) \bar{U}_{0}(k,G_{t}) } \\
  + \sum_{j=1}^{\infty} \biggl( &\bE{\bar{U}_{0}(0,G_{s}) \bar{U}_{j}(0,G_{t})} + 2 \sum_{k=1}^{\infty} \bE{\bar{U}_{0}(0,G_{s}) \bar{U}_{j}(k,G_{t})} \\
  + 2 &\sum_{l=1}^{\infty} \bE{\bar{U}_{0}(l,G_{s}) \bar{U}_{j}(0,G_{t})} + 4 \sum_{l=1}^{\infty} \sum_{k=1}^{\infty} \bE{\bar{U}_{0}(l,G_{s}) \bar{U}_{j}(k,G_{t})} \biggr) \\
  + \sum_{i=1}^{\infty} \biggl( &\bE{\bar{U}_{i}(0,G_{s}) \bar{U}_{0}(0,G_{t})} + 2 \sum_{k=1}^{\infty} \bE{\bar{U}_{i}(k,G_{s}) \bar{U}_{0}(0,G_{t})} \\
  + 2 &\sum_{l=1}^{\infty} \bE{\bar{U}_{i}(0,G_{s}) \bar{U}_{0}(l,G_{t})} + 4 \sum_{l=1}^{\infty} \sum_{k=1}^{\infty} \bE{\bar{U}_{i}(k,G_{s}) \bar{U}_{0}(l,G_{t})} \biggr) \biggr).
\end{split}  
\end{equation}
In the special case that $G_{s}=G_{t}=G$ we see that $Q_{s,t}(G,G) = \sigma_{J}^{2}(G)$. Finally, we notice that
\begin{equation*}
  \bE{\bar{U}_{0}(0,G_{s}) \bar{U}_{j}(k,G_{t})} = \bE{U_{0}(0,G_{s}) U_{j}(k,G_{t})} - \bE{U_{0}(0,G_{s})} \bE{U_{0}(k,G_{t})}
\end{equation*}
can be explicitly computed for all $j,k,l \in \N_{0}$ using \eqref{expectation_U} and a slight modification of \eqref{covariance_U} taking into account that $G_{s}$ and $G_{t}$ may differ from $G$.
\end{document}